\documentclass[preprint]{elsarticle}

\usepackage{lineno,hyperref}
\modulolinenumbers[3]

\journal{Annals of Pure and Applied Logic}

\usepackage{amsmath,amssymb,amscd,graphicx,amsthm,amsfonts,amstext,mathrsfs,upgreek, stmaryrd, setspace, multicol, mathtools, bbm, dsfont, turnstile,pb-diagram, flafter, float, mathtools, verbatim}
\graphicspath{{images/}}
\usepackage{pstcol}
\usepackage{pstricks-add}

\usepackage
[a4paper, margin=1.2in]{geometry}
\usepackage{hyperref}

\hypersetup{colorlinks=true}

\newcommand{\E}{\mathcal{E}}
\newcommand{\M}{\mathcal{M}}
\newcommand{\N}{\mathcal{N}}

\newcommand{\T}{\mathcal{T}}
\newcommand{\RR}{\mathbb{R}}

\newcommand{\OO}{\mathcal{O}}
\newcommand{\PP}{\mathbb{P}}
\newcommand{\QQ}{\mathbb{Q}}

\newcommand{\CC}{\mathbb{C}}

\newcommand{\MM}{\mathbb{M}}
\newcommand{\ML}{\mathbb{ML}}
\newcommand{\bsi}[1]{\boldsymbol{\Sigma}^{1}_{#1}}
\newcommand{\bpi}[1]{\boldsymbol{\Pi}^{1}_{#1}}
\newcommand{\bdi}[1]{\boldsymbol{\Delta}^{1}_{#1}}

\newcommand{\LLL}{\mathbb{L}}
\newcommand{\Amo}[2]{\mathbb{A}_{#1,\: #2}}
\newcommand{\amos}{\mathbb{A}_{\mathbb{S}}}

\newcommand{\qq}[1]{\mathcal{Q}(#1)}
\newcommand{\Q}{\mathcal{Q}}
\newcommand{\pc}{\mathcal{P}}

\newcommand{\itr}{\mathfrak{T}}

\newcommand{\UU}{\mathbb{V}}
\newcommand{\sa}{\mathbb{S}}
\newcommand{\hull}[2]{\text{Hull}^{#1}(#2)}

\newcommand{\Next}[1]{\text{Next}(#1)}

\newcommand{\li}{\trianglelefteq}
\newcommand{\HHH}[1]{\mathrm{H}({#1})}

\newcommand{\seq}[2]{\langle #1 \mid #2 \rangle}

\newcommand{\w}{\omega}

\newcommand{\HC}[1]{\mathrm{H}_{\omega_{#1}}}

\newcommand{\id}{\operatorname{id}}
\newcommand{\cp}[1]{\operatorname{crit}(#1)}
\newcommand{\mn}[2]{M^\sharp_{#1}(#2)}
\newcommand{\toe}[1]{E_{\rm{top}}^{\,#1}}
\newcommand{\centm}[1]{\begin{center}#1\end{center}}
\newcommand{\forc}[2]{\dststile{#2}{#1}}
\newcommand{\cof}{\operatorname{cof}}

\newcommand{\Trcl}{\operatorname{Trcl}}

\newcommand{\Spl}{\operatorname{split}}
\newcommand{\Ter}{\operatorname{Term}}
\newcommand{\stem}{\operatorname{stem}}
\newcommand{\Lev}[2]{\operatorname{Lev}_{#1}{(#2)}}
\newcommand{\succc}[2]{\mathrm{Succ}_{#1}(#2)}

\newcommand{\ult}[1]{\operatorname{Ult}(#1)}

\newcommand{\res}[2]{#1\!\upharpoonright\!{#2}}

\newcommand{\lh}{\operatorname{lh}}

\newcommand{\AD}{{\sf AD}}
\newcommand{\ZFC}{{\sf ZFC}}
\newcommand{\ZF}{{\sf ZF}}
\newcommand{\AC}{{\sf AC}}

\newcommand{\OR}{{\sf OR}}

\newcommand{\GCH}{{\sf GCH}}

\newcommand{\ran}[1]{{{\rm{rng}}(#1)}}
\newcommand{\dom}[1]{{{\rm{dom}}(#1)}}

\newcommand{\heigth}[1]{{{\rm{ht}}(#1)}}

\newcommand{\VV}{{V}}

\newcommand{\LL}{L}

\newtheorem{theorem}{Theorem}[section]
\newtheorem{lemma}[theorem]{Lemma}
\newtheorem{corollary}[theorem]{Corollary}
\newtheorem{proposition}[theorem]{Proposition}

\newtheorem{question}[theorem]{Question}

\newtheorem{claim}[theorem]{Claim}
\newtheorem*{claim*}{Claim}
\newtheorem*{subclaim*}{Subclaim}

\theoremstyle{definition}
\newtheorem{definition}[theorem]{Definition}

\newtheorem{fact}[theorem]{Fact}

\theoremstyle{remark}
\newtheorem{remark}[theorem]{Remark}

\newtheorem{case}{\scshape Case}
\newtheorem*{case*}{\scshape Case}

\definecolor{teal2}{rgb}{0.036, 0.512, 0.512}

\usepackage{hyperref}
\hypersetup{colorlinks=true}
\hypersetup{citecolor=teal2,linkcolor=teal2}

\begin{document}
\nocite{*} 

\begin{frontmatter}

\title{Preserving levels of projective determinacy by tree forcings\tnoteref{mytitlenote}}



\author[fabi]{Fabiana Castiblanco\corref{cor1}}
\ead{fabianaomega0@gmail.com}
\author[phil]{Philipp Schlicht}\ead{philipp.schlicht@bristol.ac.uk}

\cortext[cor1]{Corresponding author}
\address[fabi]{Institut f\"ur Mathematische Logik und Grundlagenforschung, Universit\"at M\"unster, 
Einsteinstra{\ss}e 62, 48149 M\"unster, Germany}
\address[phil]{School of Mathematics, University of Bristol, Fry Building, Woodland Road, Bristol, BS8 1UG, UK} 



\begin{abstract}
We prove that various classical tree forcings---for instance Sacks forcing, Mathias forcing, Laver forcing, Miller forcing and Silver forcing---preserve the statement that every real has a sharp and hence analytic determinacy. 
We then lift this result via methods of inner model theory to obtain level-by-level preservation of projective determinacy ({\sf PD}). 
Assuming {\sf PD}, we further prove that projective generic absoluteness holds and no new equivalence classes are added to thin projective transitive relations by these forcings. 
\end{abstract}

\begin{keyword}
Forcing preservation theorems, tree forcing, Projective Determinacy, thin relations, sharps, mouse operator, uniform indiscernibles\end{keyword}

\end{frontmatter}

\linenumbers

\section{Motivation}

A. Levy and R. Solovay \cite{levysol} have shown that if $\kappa$ is a measurable cardinal and $\PP$ is a small forcing notion, i.e. $|\PP|<\kappa$, then $\kappa$ remains measurable in the generic extension $\VV^\PP$.  Since the existence of compact, supercompact and huge cardinals, among others, are characterized by the existence of  certain elementary embeddings related to  ultrapowers, variants of the Levy-Solovay argument were performed in that cases showing that small forcing preserves these large cardinal properties as well. 

Other large cardinal notions are instead characterized by the existence of extender embeddings rather than simple ultrapower embeddings.  In this respect,  Hamkins and Woodin \cite{hamkins} have shown that if $\kappa$ is $\lambda$-strong then it would be also $\lambda$-strong in a generic extension obtained after forcing with a small poset.\footnote{In fact, they prove as well that a $\lambda$-strong cardinal $\kappa$ cannot be created via any forcing of cardinality $<\kappa$ (except possibly in the case of ordinals $\lambda$ with $\cof{\lambda}\leq |\PP|^{\: +}$).} Hence, the \emph{strongness} and \emph{Woodiness} of a cardinal are also preserved by small forcing. 

Many global consequences implied by the existence of large cardinals  also are preserved after forcing with certain posets.  For instance, the existence of $x^\sharp$ for each set of ordinals $x$  satisfying $\sup{x}\subset \kappa$ is a known consequence of the existence of a Ramsey cardinal $\kappa$ \cite[Chapter 2, \S 9]{kanamori}.  Moreover, from a Ramsey cardinal $\kappa$ we obtain closure under sharps for reals in the universe and also, we gain $\bsi{3}$-absoluteness for small generic extensions.   

This is closely related to the following well-known result: 

\begin{fact} 
Suppose that for every set of ordinals $x$,  $x^\sharp$ exists.  Let $\PP$ be a forcing in $V$ and let $G$ be $\PP$-generic over $V$.  Then also $\VV[G]\models \forall x\ (x^\sharp \text{ exists})$. 
\end{fact}

If we consider the property that $x^\sharp$ exists for every real $x$, this preservation result is no longer true.  In fact,  R. David (cf. \cite{david}) has shown that in the minimal model closed under sharps for reals $\LL^\sharp$ there is a $\bsi{3}$-forcing $\PP$ adding a real with no sharp in the generic extension. 

Nevertheless, if we restrict the complexity of the forcing,  we obtain positive results.  Furthermore, Schlicht \cite[Lemma 3.11]{phildis} has proved a more general statement: given $n<\w$, if $M^\#_n(x)$ exists and is $\w_1$-iterable for every $x\in\mathbb{R}$, then every $\Sigma^1_2$ provably c.c.c. forcing  preserves the existence and iterability of $M_n^\sharp(x)$.\:     Thus, it is natural to ask whether we can extend this result to the wider class of $\bsi{2}$ proper forcing notions.

This paper addresses the preservation problem above and its consequences when we consider Sacks ($\sa$), Silver ($\UU$), Mathias($\MM$), Miller ($\ML$) and Laver ($\LLL$) forcing.  These forcing notions are proper and their complexity is at most $\bdi{2}$.  We prove that for each natural number $n$, all the forcing notions in the set $\T=\{\sa, \UU, \MM, \LLL,\ML\}$ preserve $M_n^\sharp(x)$ for every $x\in{}^\w\w$ or equivalently,  every partial order in $\T$ preserves $\bpi{n+1}$-determinacy (cf. Theorem \ref{mainthm}).    As a consequence, from the existence of $\mn{n}{x}$ for every real $x$ we obtain  that $\bsi{n+3}$-$\PP$-absoluteness holds for every $\PP\in\T$ (cf. Theorem \ref{abs}).  This gives a partial answer to \cite[Question 7.3]{ikegami_2}.

With these results in hand,  we can show that given $n\in\w$, every forcing notion in $\T$ does not add any new orbits to $\bdi{n+3}$-thin transitive relations if we assume  the existence of $\mn{n}{x}$ for every real $x$ (cf. Theorem \ref{classes}).    As a motivation to this main result, we show that all the forcing notions in $\T$ do not change the value of the second uniform indiscernible $u_2$, which partially answers \cite[Question 7.4]{ikegami_2}.

\bigskip 
\noindent 
{\bf Acknowledgements.} 
This work is part of the Ph.D. thesis of the first author. She would like to thank Ralf Schindler for his permission to include the joint results in Section \ref{thinexample} and providing background on inner model theory. 
The authors would further like to thank the referee for the careful reading and numerous highly useful suggestions. 

The first author gratefully acknowledges support from the SFB 878 program ``\emph{Groups, Geometry \& Actions}'' financed by the Deutsche Forschungsgemeinschaft (DFG). 
This project has received funding from the European Union's Horizon 2020 research and innovation programme under the Marie Sk\l odowska-Curie grant agreement No 794020 (IMIC) of the second author. 
The second author was partially supported by FWF grant number I4039. 

\section{Basic notions}
Let $\RR$ denote the set of real numbers.  As usual, we identify $\RR$ with the power set $\wp(\w)$ of the set $\omega$ of natural numbers, with the Baire space ${}^\w\w$, with the Cantor space ${}^\w2$ or with the set  ${}^{\uparrow\w}\w$\index{${}^{\uparrow\w}\w$} of strictly increasing functions from $\w$ to $\w$, depending on the context. 

We assume familiarity of the reader with the basic facts about forcing.  For undefined notions, consult the texts \cite{jech} and \cite{ralfbook}.
For our purposes, a forcing\index{forcing} notion $\PP$ consists of an underlying set $P$ together with a preorder $\leq_\PP$ on $P$ and the induced incompatibility relation $\perp_\PP$.  In this case, we write $\PP=\langle P,\leq_\PP, \perp_\PP\rangle$.  We often identify the  underlying set $P$ with $\PP$ itself.

\subsection{Arboreal forcing}
\begin{definition}Let $n\geq 1$ be a natural number.  Let $M$ be a  transitive model of $\ZFC$  and let $\PP\in M$ be a forcing notion. We say that $M$ is (1-step) $\bsi{n}$-$\PP$-absolute iff for every $\Sigma^1_n$-formula $\varphi$ and for every real $a\in M$, we have \[M\models \varphi(a)\iff M^\PP\models \varphi(a).\]
This is equivalent to the expression $M^\VV\prec_{\Sigma^1_n}M^{\VV^\PP}$.  Similarly, we define (1-step) $\bpi{n}$-$\PP$-absoluteness.
\end{definition}

\begin{definition}Let $n\geq 1$ be a natural number.  Let $M$ be a transitive model of $\ZFC$.  We say that $M$ is $\bsi{n}$- correct\index{correctness! $\bsi{n}$-correctness} (in $\VV$) iff  $M\prec_{\Sigma^1_n} \VV$. In other words, for each $\Sigma^1_n$-formula $\varphi$ and every real  $a\in M$,\[M\models \varphi(a)\iff \VV\models\varphi(a).\]In a similar way, we define $\bpi{n}$-correctness.\end{definition}

\begin{definition}Let $n<\w$.  We say that a forcing notion $\PP=\langle P,\leq_\PP, \perp_\PP\rangle$ is $\bsi{n}$ if $P\subset{}^\w\w$ is $\bsi{n}$ and the order and incompatibility relations $\leq_\PP$ and $\perp_\PP$ are $\bsi{n}$-subsets of ${}^\w\w\times{}^\w\w$.  In a similar way, we define $\bpi{n}$ forcing notions.  In addition, we say that $\PP$ is $\bdi{n}$ if it is both $\bsi{n}$ and $\bpi{n}$.  Finally, we say that $\PP$ is projective if and only if $\PP$ is $\bsi{n}$ for some $n<\w$. \end{definition} 

\begin{definition}We say  that a forcing notion $\PP$  is Suslin if and only if it is $\bsi{1}$-definable.  Also, we say that $\PP$ is co-Suslin if and only if it is $\bpi{1}$-definable.  
\end{definition}

\begin{definition}Let $\PP=\langle P,\leq_\PP,\perp_\PP\rangle$ be a poset definable by a projective formula with parameter $a\in{}^\w\w$.   Let $M$ be a transitive model of $\ZF$ containing the parameter $a$.  Then, $\PP^M$, $P^M$, $\leq_\PP^M$ and $\perp_\PP^M$ denote the forcing notion $\PP$ re-interpreted in $M$.\:  Also, we say that $\PP$ is absolute for $M$ if \[\leq_\PP^M\: =\: \leq_\PP\cap\, M\:\text{ and\:  }\perp_\PP^M\: =\: \perp_\PP\cap\, M.\] 
\end{definition}

\begin{definition}Let $n\geq 1$ be a natural number.  A forcing notion $\PP$ is called \emph{provably} $\Delta^1_{n}$ if there is a $\Sigma^1_n$-formula $\varphi$ and a $\Pi^1_n$-formula $\psi$ such that\[\ZFC\vdash ``\varphi\text{ and $\psi$ define the same triple $\PP=(P,\leq_\PP,\perp_\PP)$'' }.\]\end{definition}
\begin{definition}Let $\PP$ be a forcing notion.  We say that a cardinal $\lambda$ is \emph{sufficiently large} if $\lambda>2^{|\PP|}$ and we write $\lambda\gg\PP$.\: 
\end{definition}

\begin{definition}Let $\PP$ be a forcing notion and let $\lambda\gg\PP$.  Let $M\prec \HHH{\lambda}$ be an elementary substructure.  A condition $q\in\PP$ is $(M,\PP)$-generic iff for every dense set $D\subset \PP$ such that $D\in M$, $D\cap M$ is predense below $q$.  Equivalently, $q\in \PP$ is $(M,\PP)$-generic if and only if \[q\forc{\PP}{\VV}\: \text{``$\dot{G}\cap P\cap M$ is $\PP\cap M$-generic over $M$''}.\]\end{definition}

\begin{definition} A forcing notion $\PP$ is \emph{proper } iff for any  $\lambda\gg\PP$ and for any countable elementary substructure $M\prec \HHH{\lambda}$ with $\PP\in M$, every $p\in \PP\cap M$ has an extension $q\leq_\PP p$ which is $(M,\PP)$-generic. 
\end{definition}

\begin{definition}Let $\PP\subset\RR$ be a forcing notion.\:  We say $\tau$ is a countable $\PP$-name for a real if it is a countable set of pairs $(\check{n}, p)$, where $n\in\w$ and $p\in\PP$.   
\end{definition}

If $\PP\subseteq \RR$ is proper, not every $\PP$-name for a real is in $\HC{1}$. However, modulo equivalence through conditions, we have countable names for reals:

\begin{proposition}\label{cnames}Let $\PP$ be a proper forcing notion whose conditions are real numbers.  Let $p\in\PP$ and let $\tau$ be a $\PP$-name for a real $x\in\VV^\PP$ such that $p\forc{\PP}{\VV} \tau\subset \w$.  Then,  there exists a condition $q\leq p$ and a countable $\PP$-name $\sigma$ such that $q\forc{\PP}{\VV} \sigma =\tau$.
\end{proposition}

\begin{proof}Suppose  $\tau=\{\langle \check{n}, p'\rangle:n\in\w, p'\in A_n\}$ is a $\PP$-name for the real $x$, where each $A_n$ is an antichain. 
Let $G$ be $\PP$-generic over $\VV$ containing $p$ and take $X=\{\langle \check{n}, p'\rangle\in \tau: p'\in G\cap A_n\}\subset \tau$. Since each $A_n$ is an antichain, $X$ is a countable subset in $V[G]$. Since $\PP$ is proper, by \cite[Lemma 31.4]{jech} there is countable set $Y\in V$ covering $X$.  Let $\sigma=\tau\cap Y$.  Then $\sigma\in V$ is a countable name and $\sigma^G=\tau^G=x$.  Thus, we can pick $q\in G$, $q\leq p$ such that $q\Vdash_\PP \sigma=\tau$ as required.\end{proof}

All the forcing notions which we shall consider  are strongly proper in the sense of \cite[Definition 5]{bagaria}.  In \cite[Definition 3.2.]{fsk}, this property is called \emph{properness-for-candidates}.  From now on, $\ZFC^*$ stands for some unspecified sufficiently large finite fragment of $\ZFC$.
\begin{definition}[Shelah]Let $\PP$ be a projective forcing defined by a formula with real parameter $a$.  We say that a countable transitive model $M$ of $\ZFC^*$ is a \emph{candidate} if it contains the defining parameter $a$.\end{definition}

\begin{definition}Let $\PP$ be a projective forcing defined by a formula in the real parameter $a$.  We say that $\PP$ is \emph{strongly proper} if for all candidates $M$ containing $a$ and satisfying $P^M\subseteq P$, $\leq_\PP^M\subseteq \leq_\PP$ and $\perp_{\PP^M}\subseteq \perp_\PP$, every condition $p\in P^M$ has an extension $q\leq_\PP p$ which is $(M,\PP)$-generic.  \end{definition}
Clearly, every projective c.c.c. forcing notion is strongly proper and every strongly proper forcing notion is proper.  

\begin{definition}
A strongly proper forcing notion $\PP=\langle \PP,\leq_{\PP}, \perp_{\PP} \rangle$ is called Suslin proper if $P$, $\leq_{\PP}$ and $\perp_{\PP}$ are $\bsi{1}$. 
Moreover, a forcing notion $\PP=\langle \PP,\leq_{\PP}, \perp_{\PP} \rangle$ is called Suslin$^+$ proper if
\begin{enumerate}[(i)]
\item $P$ and $\leq_{\PP}$ are $\bsi{1}$, 
\item there is a $\bsi{2}$, $(\w+1)$-ary relation $\text{epd}(p_0,p_1,\dots,q)$ such that if $\text{epd}(p_0,p_1,\dots,q)$ holds for $p_i,q\in\PP$, then $\{p_i:\i<\w\}$ is predense below $q$ and
\item for every candidate $M$ containing all relevant parameters, and all $p\in\PP^M$ there is $q\leq p$ such that for every $D\in M$ which is $\PP^M$-dense, there exists an enumeration $\{d_i:i<\w\}\subset D$ such that $\text{epd}(d_0,d_1,\dots,q)$ holds.  In this case we say that $q$ is an effective $(M,\PP)$-generic condition, and we call this property \emph{effective}\,-properness for candidates.
\end{enumerate}
\end{definition}

Mathias forcing is Suslin proper, while Sacks, Silver, Laver and Miller forcing are not. 
This was Goldstern's and Shelah's motivation for introducing the weaker notion Suslin$^+$ proper, which covers these forcings.  
Note that Suslin proper implies Suslin$^+$ proper implies proper. 

The following lemma will be crucial for our absoluteness arguments.
\begin{lemma}\label{fordef}Let $\PP$ be a Suslin$^+$ proper forcing notion.  Let $p\in\PP$ and let $\tau$ be a countable $\PP$-name for a real.  Then, for all $n\geq2$:
\begin{enumerate}[(1)]\item If $\varphi(x)$ is $\bpi{n}$, then $p\Vdash\varphi(\tau)$ is $\bpi{n}$.\footnote{When we write that $\varphi(x)$ is ${\bf \Pi}^1_n$, we mean that it is a $\Pi^1_n$-formula with the free variable $x$ and real parameters.} 
 \item If $\varphi(x)$ is $\bsi{n}$, then $p\Vdash\varphi(\tau)$ is $\bsi{n+1}$.
\end{enumerate}
\end{lemma}

\begin{proof} See \cite[Lemma 3.7.]{fsk}.\end{proof}
\begin{definition}A partial ordering $(T,<)$ is called a tree if for every $s\in T$, the set $T_{<s}=\{t\in T: t<s\}$ is well-ordered.\:
\begin{enumerate}[(a)]
\item If $s\in T$, $\succc{T}{ s}$ is the set of all immediate successors of $s$ in $T$.  We say that $T_s=\{t\in T: $\:either $s=t$, $t< s$ or $s< t \}$ is the subtree determined by $s$.

\item The set of splitting points of $T$ is defined by $\Spl(T)=\{t\in T: |\succc{T}{t}|>1\}$.  The least node which splits in $T$ is denoted by $\stem(T)$.

\item The $n^{\text{th}}$ \emph{splitting level} of $T$ is the set \begin{center}$\Lev{n}{T}=\{t: t\in\Spl(T) \text{ and } |\{s\subsetneq t: s\in\Spl(T)\}|=n\}.$\end{center}

\item Let $\Ter(T)=\{s\in T: \succc{T}{s}=\varnothing\}$.
\end{enumerate}

A tree $T$ is called \emph{perfect} if for every $s\in T$ there exists $t\in\Spl(T)$ such that $s<t$ and \emph{superperfect} if for every node $s\in T$ there exists $t\in\Spl(T)$ such that $s<t$ and $\succc{T}{t}$ is infinite.
\end{definition}

\begin{definition}\label{arb}A partial order $\PP$ is \emph{arboreal} if its conditions are perfect trees on $\w$ or $2$ ordered by reversed inclusion ($T\leq S$ if $T\subset S$).\:  A partial order $\PP$ is \emph{strongly arboreal} if it is arboreal and for all $T\in \PP$, if $s\in T$ then $T_s\in \PP$.\end{definition}

If $\PP$ is strongly arboreal, we can associate a unique sequence $x_G$ to each $\PP$-generic filter $G$ over $V$ in the following way: \[x_G=\bigcup\{\stem(T): T\in G\}=\bigcap\{[T]: T\in G\}.\]  Notice that $x_G$ is an element of ${}^\w2$ or ${}^\w\w$ and $G=\{T\in\PP:  x_G\in[T]\}$ if $\PP$ is assumed to be proper. (This follows from \cite[Proposition 2.18]{ikegami_2}.) 
Therefore, $V[G]=V[x_G]$.  In this setting, we refer to the real $x_G$ as a $\PP$-generic real over $\VV$.

\subsubsection{Sacks forcing $\sa$}
The perfect set forcing $\sa$ was introduced by Sacks in \cite{sacksf}, where he constructed a model of $\ZFC$ with exactly two degrees of constructibility.  His results have been remarkable for further development of forcing as well as for recursion theory.
 
\begin{definition}  Sacks forcing $\sa$ is defined in the following way: 
\[\sa= \{T: T \text{\: is a perfect subtree of ${}^{<\w}2$}\}.\]
For $S, T\in \sa$ we stipulate $S\leq T$ if and only if $S\subset T$. 
\end{definition}

It is clear that $\sa$ is an strongly arboreal forcing notion.\:  On the other hand, $\sa$ does not satisfy the c.c.c.  In fact, there are antichains of size $2^{\aleph_0}$: \:  Let $\{A_\alpha:\alpha<2^{\aleph_0}\}$ be an almost disjoint family of subsets of $\w$ and for each $\alpha<2^{\aleph_0}$ choose a perfect tree $T_\alpha$ whose splitting levels are exactly the elements of $A_\alpha$, v.g., $T_\alpha=\{s\in2^{<\w}:\forall n <|s|(n\notin A_\alpha \to s(n)=0)\}$.\:  If $\alpha<\beta$, then $T_\alpha\cap T_\beta$ includes no perfect tree, so they are incompatible.\:  Nevertheless, $\sa$ does not collapse $\aleph_1$ either since it satisfies Baumgartner's Axiom $\sf A$ (see \cite[Lemma ]{judah})

An important property of Sacks forcing is its minimality. Abstractly, we say that a real $x$ is minimal over a model $M$ if $x\notin M$ and every real $y\in M[x]$ is in $M$ or reconstructs $x$, i.e. $x\in M[y]$.

\begin{theorem}[Sacks]Suppose that $s$ is a Sacks real over $V$.  Then, in $V[s]$ for any set $X\notin V$ and $X\subseteq V$, $V[X]=V[s]$. \end{theorem} 

\subsubsection{Silver forcing $\UU$}\index{forcing notions!Silver $\UU$}
\begin{definition} A uniform tree $T\subseteq {}^{<\w}2$ is a perfect tree such that for all $s, t\in T$ of the same length we have \[s^\frown0\in T  \text{\: iff \: }t^\frown0\in T \text{\: and\: } s^\frown 1\in T \text{\: iff \:} t^\frown 1\in T.\]\end{definition}

Let  $\UU$ be the collection of all uniform trees with infinitely many splitting levels. For $S, T\in \UU$ we stipulate  $S\leq T$ if and only if $S\subset T$.\:  The poset $\UU$ is known as Silver forcing and if $G$ is generic, $x_G$ defined as in \ref{arb} is called a Silver real.  Basically, Silver forcing is the uniform version of Sacks forcing.

Neither a Silver real occurs in a Sacks forcing extension nor a Sacks real occurs in a Silver extension.  However, like Sacks forcing, Silver forcing adds a real of minimal degree of constructibility (see \cite[Theorem 4.1]{gri} or \cite[Theorem 18]{halbeisen2}).

\subsubsection{Mathias forcing $\MM$}

We follow the terminology of \cite[Chapter 10]{halbeisen}. 

\begin{definition} \index{filter!free} Let $\mathcal{F}$ be a filter over $\w$.  We say that $\mathcal{F}$ is a \emph{free filter} if it contains the Frechet filter.
In this case, let \[\mathcal{F}^+=\{x\subseteq\w: \forall z\in\mathcal{F}(|x\cap z|=\w)\}.\]
\end{definition}

\begin{definition}A family $\mathcal{E}$ of subsets of $\w$ is called a  \emph{free family}\index{family!free} is there is a free filter $\mathcal{F}\subseteq [\w]^\w$ such that $\mathcal{E}=\mathcal{F}^+$. \end{definition}
\noindent Notice that a free family does not contain finite sets and is closed under supersets.\: 

\begin{definition} Let $\mathcal{E}$ be a free family.  We define a game $G_\mathcal{E}$\index{game!{\sc Maiden} and {\sc Death} $G_\mathcal{E}$} between two players, {\sf Adam} and {\sf Eve}, as follows: 
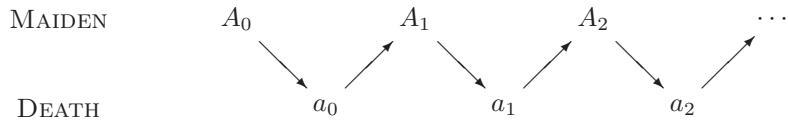
\[\begin{diagram}\node{\text{\sf Adam}}  \node[2]{A_0}\arrow{se,t}{}\node[2]{A_1}\arrow{se,t}{}\node[2]{A_2}\arrow{se,t}{}\node[2]{\cdots}\\
\node{\text{\sf Eve}}
\node[3]{a_0}\arrow{ne,t}{}\node[2]{a_1}\arrow{ne,t}{}\node[2]{a_2}\arrow{ne,t}{} \end{diagram}\]
\noindent where $A_i\in\mathcal{E}$,  $a_i\in A_i$ and further $A_{i+1}\subseteq A_i$ and $a_i< a_{i+1}$, for each $i\in\w$.  We say that {\sf Eve} wins the game $G_\mathcal{E}$ if and only if $\{a_i:i<\w\}\in\mathcal{E}$.\end{definition}

\begin{definition} A family $\mathcal{E}$ of subsets of $\w$ is \emph{Ramsey}\index{family!Ramsey} if the {\sf Adam} has no winning strategy in the game $G_\mathcal{E}$, i.e., {\sf Eve} can always defeat any given strategy of the {\sf Adam}.\:  Note that this does not imply that {\sf Eve} has a winning strategy.
\end{definition}

The notion of Ramsey families generalises Ramsey ultrafilters (see \cite[Chapter 10]{halbeisen}). 

\begin{definition}Let $\E$ be a Ramsey family.  We define:\[\MM_\E=\{(s,A): s\in[\w]^{<\w}, A\in\E, \max(s)<\min(A)\}.\]In $\MM_\E$,\index{forcing notions!Mathias $\MM_\E$, $\MM$} we stipulate the following order:
\[(s,A)\leq (t,B)\iff t\subseteq s,\: A\subseteq B \text{\: and\: } s\smallsetminus t\subseteq B.\]
The poset $(\MM_\E,\leq)$ is called \emph{Mathias forcing} with respect to $\E$.\:  If $G$ is $\MM_\E$-generic over $V$, we say that  \[m_G=\bigcup\{s: (s,A)\in G\:\text{for some\:}A\}\] is a \emph{Mathias real}.  \:   When $\E=[\w]^\w$ we write $\MM:=\MM_\E$ and we refer to $\MM$ as Mathias forcing.   
\end{definition}

The main property of $\MM_\E$ is the next so called \emph{Mathias property}\index{forcing!notion!Mathias property}. 

\begin{lemma}[Mathias]\label{mathiasproperty} 
Every infinite subset of an $\MM_{\E}$-generic real over $V$ is also $\MM_{\E}$-generic over $V$.\end{lemma}
\begin{proof}The original source of this result is \cite[Corollary 2.5]{mathias}.  For a proof in terms of the game $G_\E$ see \cite[Proposition 11]{halbeisen2}. \end{proof}


Mathias forcing $\MM$ can be represented also as an arboreal forcing notion.\:   This characterization will be useful through section \ref{sharpsreals}.  Recall that ${}^{\uparrow<\w}\w$ stands for the set of all strictly increasing finite sequences of natural numbers. 

\begin{proposition}\label{mathiastree2}\index{forcing notions!Mathias $\MM_\E$, $\MM$!max@$\MM_{\text{tree}}$} For each $(s,A)\in\MM$, let \[T_{(s,\, A)}:=\{t\in{}^{\uparrow<\w}\w:\:\:  s\subset\ran{t}\subset s\cup A\}.\]Let $\MM_{\text{tree}}=\{T_{(s,A)}: (s,A)\in\MM\}$.   Then $\pi: \mathbb{M}\mapsto \MM_{\text{tree}}$ defined as $\pi(s,A)=T_{(s,A)}$ is an isomorphism, where $\MM_{\text{tree}}$ is ordered by $T\leq_{\MM_\textrm{tree}}S$ iff $T\subset S$.
\end{proposition}

\begin{proof}It is enough to see that $\pi$ is order-preserving, because $\pi$ is clearly onto.  Suppose that $(s,A)\leq_\MM (t,B)$.  Let $p\in T_{(s,\; A)}$; then, $s\subseteq \ran{p}\subseteq s\cup A$.  As $t\subseteq s$, $s\smallsetminus t\subseteq B$ and $A\subseteq B$ we have that   \[t\subseteq \ran{p}\subseteq s\cup A= t\cup (s\smallsetminus t) \cup A\subseteq t\cup B.\]Thus, $p\in T_{(t,\, B)}$.  Therefore, $T_{(s,\;A)}\leq_{\MM_{\text{tree}}} T_{(t,\; B)}$.

On the other hand, suppose that $T_{(s,\;A)}\leq_{\MM_{\text{tree}}} T_{(t,\; B)}$ but $(s,A)\not\leq_{\MM} (t,B)$.  We shall consider three cases:

\begin{case} Suppose that $t\not\subseteq s$.   Let $p:|s|\to s$ be a strictly increasing enumeration of $s$.  As $s= \ran{p}\subseteq s\cup A$, $p\in T_{(s,\, A)}$.   Since $t\nsubseteq s$, then $t\nsubseteq \ran{p}$ and therefore $p\notin T_{(t,\; B)}$.\end{case}

\begin{case} Suppose $t\subseteq s$ but $s\smallsetminus t\nsubseteq B$.  Again, take $p: |s|\to s$ be a strictly increasing enumeration.  Let $m=\min ((s\smallsetminus t)\smallsetminus B)$.  Clearly $m\in\ran{p}$ but $m\notin t\cup B$.  So, $t\subset\ran{p}$ but $\ran{p}\nsubseteq t\cup B$.  Thus, $p\notin T_{(t,\; B)}$.\end{case}

\begin{case}Suppose that $t\subseteq s$, $s\smallsetminus t\subset B$ but $A\nsubseteq B$.  Let $m=\min(A\smallsetminus B)$.   Let $p:|s|+1\to s\cup \{m\}$ be a strictly increasing enumeration.  Note that, since $\max(s)<\min(A)$, $m\notin s$ and therefore, $m\notin t$.  Since $s\subset\ran{p}\subset s\cup A$, $t\in T_{(s,\, A)}$.  So, $t\subset\ran{p}$ but $m\in\ran{p}$ and $m\notin t\cup B$.  Thus, $p\notin T_{(t,\, B)}$.\end{case}
From cases 1), 2) and 3) it follows that $T(s,\,A)\nsubseteq T(t,\,B)$ which contradicts our assumption.  Therefore, we must have $(s,A)\leq_{\MM} (t,B)$ as desired. \end{proof}

We can also see $\MM$ as an arboreal forcing whose conditions are subtrees of ${}^{<\w}2$.  Given a subset $a\subseteq n$, let $\chi_a^n$  denote the characteristic function of $a\cap \{0,\dots,n-1\}=a\cap n$ with domain $n$. 

\setcounter{case}{0}
  \begin{proposition}\label{mathiastree} \index{forcing notions!Mathias $\MM_\E$, $\MM$!mbx@$\MM'_{\text{tree}}$}
 For each $(s,A)\in\MM$, let \[U_{(s,\, A)}:=\{\res{\chi_a^n}{k}:s\subset a\subset s\cup A,\: a\in[\w]^{<\w}, k\leq n\in\w\}.\]Let\: $\MM'_{\text{tree}}=\{U_{(s,\,A)}: (s,A)\in\MM\}$.  Then, $\pi^*: \mathbb{M}\mapsto \MM'_{\text{tree}}$ defined as $\pi^*(s,A)=U_{(s,\;A)}$ is an isomorphism, where $\MM'_{\text{tree}}$ is ordered by reversed inclusion.
\end{proposition} 
\begin{proof}   

We show that $\pi^*$ is order preserving.  Note that for all $u\in U_{(s,A)}$ and all $k\leq |s|,|u|$,\:  $\res{u}{k}= \res{s}{k}$.  

 First, suppose that $(s,A)\leq (t,B)$.  Let $a\in[\omega]^{<\omega}$, $n\in\omega$ such that $\chi_a^n\in U_{(s,\; A)}$.\:  As $s\subseteq a\subseteq s\cup A$ and $(s,A)\leq(t, B)$ we have $t\subseteq s\subseteq a$  and  \[a\subseteq s\cup A\subseteq s\cup B=(s\cap t)\cup (s\smallsetminus t)\cup B\subseteq (t\cup B).\]  Therefore $\chi_a^n\in U_{(t,\, B)}$.  This proves $U_{(s,\;A)}\subseteq U_{(t,\,B)}$, i.e.  $U_{(s,\;A)}\leq U_{(t,\;B)}$. 

On the other hand, suppose that $U_{(s,\,A)}\leq U_{(t,\,B)}$ but $(s,A)\not\leq (t,B)$.  We shall consider two cases:

\begin{case} Suppose that $t\not\subseteq s$.   Let $k=\min (t\setminus s)$. We have $\chi_s^{|s|}\in U_{(s,\, A)}$. We show that $\chi_s^{|s|}\notin U_{(t,\,B)}$.
First, suppose that $k<\max(s)$.\:  Then $\chi_s^{|s|}(k)\neq \chi_t^{|t|}(k)$, so $\chi_s^{|s|}\notin U_{(t,\,B)}$. 
Now, suppose that $\max(s)\leq k$. Aditionally, assume that $s\not\subseteq t$. 
Let $l=\min(s\setminus (t\cap k))$.  Then $\res{\chi_s^{|s|}\!\!}{l}\neq \res{\chi_t^{|t|}}{l}$, so $\res{\chi_s^{|s|}}{(l+1)}\neq \res{\chi_t^{|t|}}{(l+1)}$. 
Since all elements of $U_{(t,\,B)}$ are comparable with $\chi_t^{|t|}$, $\res{\chi_s^{|s|}} {(l+1)}\notin U_{(t,\,B)}$. 
To conclude, assume that $s\subseteq t$. Then $s= t\cap k$ an so $\chi_s^{k+1}\in U_{(s,\,A)}$.\: However, $\chi_s^{k+1}\neq \chi_t^{k+1}$ and hence $\chi_s^{k+1}\notin U_{(t,\,B)}$. In both subcases, we conclude then that $U_{(s,\,A)}\nleq U_{(t,\,B)}$ which contradicts our assumption.
\end{case}

\begin{case} Suppose that $t\subseteq s$ and $s\setminus t\not \subseteq B$.  Then, if $n=\min((s\setminus t)\setminus B)$ we have that $\chi_s^{n+1}\in U_{(s,\,A)}$ but $\chi_s^{n+1}\notin U_{(t,\,B)}$ which contradicts  $U_{(s,\,A)}\leq U_{(t,\,B)}$.
\end{case}From cases 1) and 2) we conclude $(s,A)\leq_\MM(t,B)$, as required.
 \end{proof} 

\subsubsection{Laver forcing $\mathbb{L}$}The forcing notion $\LLL$ was introduced by Richard Laver in \cite{laver} with the purpose of producing a model of the Borel conjecture.  This forcing notion adds a Laver real $f:\w\to\w$ which dominates all ground model functions, that is, for any $g\in{}^\w\w\cap\VV$,  $g(n)<f(n)$ for all but for finitely many $n\in\w$.  Instead of considering perfect subtrees of ${}^{<\w}2$, Laver considered infinitely splitting subtrees of ${}^{<\w}\w$. 
 
\begin{definition}\index{tree!Laver}A tree $T\subseteq ^{<\w}[\w]$ is called a \emph{Laver tree} if it has a stem $s$ and above this stem it splits into infinitely many successors at every node i.e. $ \forall t\in T(t\subseteq s$ or $|\succc{T}{t}|=\w)$.\:  \emph{Laver forcing}, denoted by $\LLL$\index{forcing notions!Laver $\LLL$}, is the set of all Laver trees ordered by inclusion.\end{definition}
\noindent If $G$ is a $\LLL$-generic filter, $x_G=\bigcup\{\stem(T):T\in G\}$ is called a \emph{Laver real}.\:  Laver reals are also of minimal degree of constructibility (cf. \cite{gray}).  

\subsubsection{Miller forcing $\mathbb{ML}$}
The rational perfect set forcing $\mathbb{ML}$ was introduced by Miller in \cite{miller}.  This forcing notion is half way between Sacks and Laver forcing.   


\begin{definition}A tree $T\subset {}^{<\w}\w$ is called superperfect if for all $s\in T$ there exists $t\supset s$ such that for infinitely many $n<\w$, $t^\frown \langle n\rangle\in T$.  Miller forcing $\ML$\index{forcing notions!Miller $\ML$} is the collection of all superperfect trees.\:  As in the other tree forcing notions, $T\leq S$ if and only if $T\subset S$.
\end{definition}

As in the case of Sacks and Laver forcing, forcing with $\ML$ produces a real of minimal constructibility degree (cf. \cite{miller}).

\begin{definition}We set \[\T=\{\sa,\UU,\MM,\LLL,\ML\}\]for the collection of the tree forcing notions presented so far.\index{tree forcing notions $\T$}
\end{definition}
 
\begin{definition}\label{parb}Let $\PP\in\T$ and let $n\in \w$.  Given two trees $T, S\in\PP$, we say $T\leq_n S$ if and only if $T\leq_\PP S$ and $\Lev{n}{S}= \Lev{n}{T
}$, i.e. the first $n$ levels of $\Spl(S)$ are still in $\Spl(T)$.     \:  We also say that $\seq{T_n}{n\in\w}\subset\PP$ is a \emph{fusion sequence} if for all $n\in\w$, $T_{n+1}\leq_n T_n$.  \end{definition}

The next statement is a concrete version of the fusion condition of Axiom A \cite[Definition 31.10 (i)]{jech}. 

\begin{lemma}[Fusion]Let $\PP\in\T$. \index{tree forcing notions $\T$!fusion}  If \: $\seq{T_n}{n<\w}$ is a fusion sequence for $\PP$ then its fusion  $T =\bigcap_{n<\w} T_n $ is an element of\: $\PP$.\:\:  Furthermore,  $T\leq_nT_n$ for each $n<\w$.
\end{lemma}

\begin{theorem} Each forcing notion $\PP\in\T$ with the relations $\leq_n$ above satisfies Axiom A. 
\end{theorem}
\begin{proof}For Sacks forcing see \cite[Lemma 7.3.2]{judah} or \cite[Lemma 25]{gesc}.  The argument for Silver forcing is completely analogous.  For Mathias forcing, see \cite[\S 7.4.A]{judah}.  For Laver and Miller forcing, see \cite[Lemmas 7.3.27 and 7.3.44]{judah}.\end{proof}

\subsection{Inner model theory}\label{imt}
For undefined notions in inner model theory appearing in this paper, we refer the reader to \cite{ralfbook},\cite{steel} and \cite{fsit}.

\begin{definition}An \emph{inner model}\index{inner models} is a transitive $\in$-model of $\ZF$ (or $\ZFC$) that contains all ordinals. \end{definition}
G\"{o}del's constructible universe $L$ is the simplest example of an inner model; also the relativization approaches of $L$ to sets of ordinals proposed by  Andr\'{a}s Hajnal and Azriel L\`{e}vy\footnote{See \cite[Chapter 1, \S 3]{kanamori}.} are instances of inner models.
\begin{definition}\label{L}Let $A$ be a set. \begin{enumerate}[(i)]\item (Hajnal)\:\: The \emph{constructive closure}\index{inner models!constructive closure $L(A)$} of $A$, $\LL(A)$ is the smallest inner model $\langle M;\in\rangle$ such that $A\in M$.  
\item (L\`{e}vy)\:\: The \emph{G\"{o}del constructive universe relative to $A$}, $\LL[A]$ is the smallest inner model $\langle M; \in, A\rangle$ which is amenable with respect to the unary predicate ``$A$'', i.e. for every $x\in M$, $x\cap A\in M$.  \end{enumerate}

\end{definition}

\begin{remark}In general, for a given set $A$ the models $L(A)$ and $L[A]$ might  have different properties.  For example, $\LL(A)$ is a model of $\ZF$ and, it might be the case that $\LL(A)\models``\AC\text{ fails ''}$\footnote{For instance, if $A=\RR$ and we assume that there are infinitely many Woodin cardinals,  we have that $L(\RR)\models \AD$ (cf. \cite[Theorem 8.3]{woodin}) and, therefore, in $L(\RR)$ does not hold $\AC$.}.  On the other hand, $\LL[A]$ is always a model of $\ZFC$ and, moreover, if $A\subset \w_1$ then $\LL[A]\models\GCH$.  Also, in general it might happen that $A\notin L[A]$\footnote{For instance, if $A=\RR$, then $L[\RR]=L$ and $\RR\notin L$ in general.}.  For more facts about these two constructions see  \cite[Proposition 3.2]{kanamori} \end{remark}

\subsubsection{Large cardinals and elementary embeddings} 

\begin{definition}\index{elementary embedding}Let $M$ and $N$ be inner models.  We say that $j:M\to N$ is an \emph{elementary embedding } if for every formula $\varphi(x_1,\dots, x_k)$ and every $a_1,\dots, a_k\in M$, we have\[M\models\varphi(a_1,\dots,a_k)\: \iff\: N\models\varphi(j(a_1),\dots,j(a_k)).\]  \end{definition}

\begin{definition}[Derived extenders]\index{extender!derived}\label{ext}Let $j:M\to N$ be an elementary embedding with $\cp{j}=\kappa$, where $M$ is an inner model of $\ZFC^-$ and $N$ is transitive and rudimentary closed.  Suppose that $N\models \lambda\leq j(\kappa)$.  For each $a\subseteq [\lambda]^{<\w}$ and $X\subseteq [\kappa]^{|a|}$, $X\in M$, we set\[X\in E_a\iff a\in j(X).\]Then, we say that $E:=E^\lambda_j=\{(a,X): X\in E_a\}$ is the $(\kappa,\lambda)$-\emph{extender}\index{extender!kappx@$(\kappa,\lambda)$-extender} derived from $j$ and we  $\kappa=\cp{E}$ for the critical point of $E$ and $\lambda=\lh E$\index{extender!length of} for the length of $E$.\end{definition}
\begin{remark}The notion of extender can be defined without considering an elementary embedding as in \cite[Definition 4.7]{stn}.  Nevertheless, it is possible to show that each extender in this new sense can be derived from an elementary embedding. In this case, $j_E$ denotes the elementary embedding associated to the extender $E$. \end{remark}

Each $E_a$ as in definition \ref{ext} is an $M$-$\kappa$-complete ultrafilter on $\wp([\kappa]^{|a|})^M$.  Therefore, we can form $\ult{M,E_a}$  in the usual way (cf. \cite[\S 2.1]{steel}).  Also, if we set $k_a([f])= j(f)(a)$ for every function $f$ with domain $[\kappa]^{|a|}$, the following diagram commutes: 

\[\begin{diagram}
\node{M}
\arrow[2]{e,t}{j}
\arrow{ese,b}{i}
\node[2]{N}
\\
\node[3]{\ult{M,E_a}}
\arrow{n,r}{k_a}
\end{diagram}
\]\medskip

Notice that $k_a([id])=a$ and $\ran{k_a}= \hull{N}{\ran{j}\cup a}$.  

\noindent Now, let $a\subseteq b\subseteq[\lambda]^{<\w}$.  Then, set $i_{ab}(x):= k^{-1}_b(k_a(x))$.  Since $[\lambda]^{<\w}$ is a directed set under inclusion and for every $a\subseteq b\subseteq c$ we have $i_{ac}=i_{bc}\circ i_{ab}$, we can define\[\ult{M,E}=  \text{dir lim}\{ \ult{M,E_a}: {a\in[\lambda]^{<\w}}\}\]where we take the direct limit under the embeddings $i_{ab}$.

\noindent Also, we can piece together the embeddings $k_a$ into $k:\ult{M, E}\to N$ given by $k(i_{a,\infty}(x))=k_a(x)$, where $i_{a,\infty}:\ult{M,E_a}\to \ult{M,E}$ is the direct limit embedding. 

As before, $\ran{k}=\hull{N}{\ran{j}\cup\lambda}$.  Therefore, $\res{k}{\lambda}=$id.  If $\ult{M, E}$ is well-founded, then $E$ is called a $(\kappa,\lambda)$-\emph{extender}.

\begin{figure}[h]\[\begin{diagram}
\node{M}
\arrow[2]{e,t}{j}
\arrow{ese,b}{i^M_{E}}
\arrow{sse,b}{i^M_{E_a}}
\node[2]{N}
\\
\node[3]{\ult{M,E}}\arrow{n,r}{k}
\\
\node[2]{\ult{M, E_a}}\arrow{ne,r}{i_{a,\infty}}
\end{diagram}
\]
\end{figure}

\begin{remark}The requirement $\lambda\leq j(\kappa)$ is not really needed.  The extenders satisfying such a property are called \emph{short extenders}\index{extender!short}.  \end{remark}

A \emph{$V$-extender} is an extender that is derived from an elementary embedding $j\colon V\rightarrow N$ with $N$ transitive. 

\begin{definition}\label{stre}\index{extender!strength of a}Let $E$ be a $V$-extender; then the \emph{strength} of $E$ is defined as \[\nu(E):= \sup\{\alpha: V_\alpha\subseteq \ult{V,E}\}.\]  \end{definition}
\noindent If $E$ is an $V$-extender, $E\notin \ult{V,E}$ and hence $\nu(E)\leq\lh(E)$ . 

\begin{definition}Let $\lambda$ be an ordinal.  We say that an uncountable $\kappa$ is $\lambda$-\emph{strong}\index{cardinal!$\lambda$-strong} if there is an elementary embedding $j:V\to M$ where $M$ is a transitive class, $\cp{j}=\kappa$ and $V_\lambda\subseteq M$. Equivalently,  $\kappa$ is $\lambda$-strong if there is a $(\kappa,\lambda)$-extender $E$ such that $V_\lambda\subseteq\ult{V,E}$ and $j_E(\kappa)\geq \lambda$. Under this setting, we say that $\kappa$ is \emph{strong}\index{cardinal!strong} if it is $\alpha$-strong for every ordinal $\alpha$. 
 \end{definition}
\begin{definition}Let $\alpha<\delta$ be ordinals and $A\subset V_\delta$. We say that an uncountable cardinal $\kappa$ is $\alpha$-$A$-\emph{reflecting}\index{cardinal!$\alpha$-$A$-reflecting} if there is an elementary embedding $j:V\to M$ witnessing that $\kappa$ is $\alpha$-strong and further  $A\cap V_\alpha=j(A)\cap V_\alpha$.  Moreover, we say that $\kappa$ is $A$-\emph{reflecting}\index{cardinal!$A$-reflecting} in $\delta$ if it is $\alpha$-$A$-reflecting for every $\alpha<\delta$. \end{definition}
\begin{definition}An uncountable regular cardinal $\delta$ is \emph{Woodin}\index{cardinal!Woodin} if and only if for every $A\subset V_\delta$, there is some $\kappa<\delta$ such that $\kappa$ is $A$-reflecting in $\delta$.  Equivalently, $\delta$ is Woodin if and only if for every function $f:\delta\to \delta$ there is $\kappa<\delta$ and an extender $E$ with $\cp{E}=\kappa$ such that ${j_E(f)(\kappa)}\leq \nu(E)$.\end{definition}

\subsubsection{Mice and iteration strategies}
\begin{definition}Let $X$ be a set of ordinals.  A $X$-\emph{potential premouse}\index{premouse!potential}\:  is an amenable structure of the form \[\M=\langle J^{\vec{E}}_\alpha(X);\in,X,\res{\vec{E}}{\alpha}, E_\alpha\rangle,\]where $\vec{E}$ is a fine extender sequence\footnote{See \cite[Definition 2.4]{steel}.} over $X$.  In this case, $\toe{\M}=E_\alpha$ denotes the top extender of $\M$.

 \noindent A $X$-potential premouse $\M=\langle J^{\vec{E}}_\alpha(X);\in,X,\res{\vec{E}}{\alpha}, E_\alpha\rangle$ is called a $X$-\emph{premouse}\index{premouse}\: if every proper initial segment of $\M$ is $\w$-sound\footnote{See \cite[Definitions 2.16-2.17]{steel}}.
 
\noindent For each $\beta\leq\alpha$, let \[\M\parallel_\beta:= \langle J^{\vec{E}}_\beta(X);\in,X,\res{\vec{E}}{\beta}, E_\beta\rangle,\]\[\M|_\beta:= \langle J^{\vec{E}}_\beta(X);\in,X,\res{\vec{E}}{\beta}, \varnothing\rangle.\]
\noindent A $X$-premouse $\pc$ is an \emph{initial segment}\index{premouse!initial segment of a} of $\M$ if there is some $\beta\leq\alpha$ such that $\pc=\M\parallel_\beta$.  In this case we write $\pc\li\M$.   Furthermore, we say that the $X$-premice $\M$ and $\pc$ are \emph{compatible}\index{premouse!compatible} if $\M\li\pc$ or $\pc\li\M$.  \end{definition}

\begin{definition}Let $\alpha<\w$.  A tree   on $\alpha$\index{iteration!tree $\itr$!order $<_\itr$} is a partial ordering $<_\itr$ with least element $0$ such that for every $\gamma<\alpha$ we have 
\begin{enumerate}[(i)]\item if $\beta<_\itr\gamma$ then $\beta<\gamma$,
\item $\{\beta:\beta<_\itr \gamma\}$ is well-ordered by $<_\itr$,\item $\gamma$ is a $<_\itr$-successor if and only if it is a successor ordinal and
\item if $\gamma$ is a limit ordinal, then $\{\beta:\beta<_\itr \gamma\}$ is $\in$-cofinal in $\gamma$.
\end{enumerate}We say that $b\subseteq \alpha$ is a \emph{branch} in $T$ if $b$ is  $<_\itr$-downward closed and $<_\itr$-well-ordered.   \end{definition}
Let $\M$ be an $\w$-sound $X$-premouse, and let $\theta$ be an ordinal.  The iteration game $\mathcal{G}_\theta(\M)$\index{iteration!game $\mathcal{G}_\theta(\M)$} is a two player game of length $\theta$ \footnote{See \cite[\S 3.1]{steel} for further details.} whose players produce:
\begin{enumerate}[(i)]\item a tree order $<_\itr$ on $\theta$,
\item a sequence of $X$-premice $\langle\M^\itr_\alpha:\alpha<\theta\rangle$ beginning with $\M^\itr_0=\M$,
\item a sequence of extenders $\langle E^\itr_\alpha:\alpha<\theta\rangle$ such that $E_\alpha$ is appearing on the $\M^\itr_\alpha$-extender sequence and 
\item a sequence $D^\itr$ of commuting embeddings  $e_{\alpha,\zeta}:\M^\itr_\alpha\to \M^\itr_\zeta$ defined for each $\alpha<_\itr\zeta$ with the next properties:\begin{enumerate}[(a)]\item If $\alpha=\text{pred}_\itr(\zeta+1)$, then $e_{\alpha, \zeta+1}= j_{E_\zeta}:\M_\alpha\to \ult{\M_\alpha, E_\zeta}$. 
\item Whenever $\gamma<\theta$ is a limit ordinal, $\M_\gamma=\text{dir}\lim \{\M_\alpha:\alpha<_\itr \gamma\}$ under the embeddings $e_{\alpha,\zeta}$ and $e_{\alpha,\gamma}$ is the direct limit embedding.
\end{enumerate}
\end{enumerate}
Player I plays all successor stages while player II does it at limit stages:
\begin{enumerate}[(1)]
\item Suppose we are at stage $\alpha+1<\theta$ during the game; the players have constructed $\res{<_\itr}{\alpha+1}$,  the sequence of premice $\langle \M^\itr_\beta:\beta\leq \alpha\rangle$, the sequence of extenders $\langle E^\itr_\beta:\beta< \alpha\rangle$ and $D^\itr\cap{\alpha+1}$.  At this move, player I should pick an extender $E_\alpha^\itr$ from the $\M^\itr_\alpha$-sequence such that for every $\xi<\alpha$,  $\lh(E^\itr_\xi)<\lh(E^\itr_\alpha)$ .  If this is not possible, then the game is over and player I loses.  Otherwise, let $\beta\leq\alpha$ be least with $\cp{E^\itr_\alpha}\in[\cp{E_\beta},\lh(E_\beta))$.   Then, it makes sense to apply  $E^\itr_\alpha$ to $\M_\beta$.  Player I then sets $\beta=\text{pred}_\itr(\alpha+1)$ and $\M_{\alpha+1}=\ult{\M_\alpha, E_\alpha}$.
\item Now, for the limit stages $\gamma$, player II picks a cofinal branch $b$ in $\res{<_\itr}{\gamma}$ and sets $\M_\gamma:=\text{dir}\lim \{\M_\alpha:\alpha\in b\}$, where the direct limit is with respect to the embeddings $e_{\alpha,\zeta}$. If player II fails to do this, then player I wins. \end{enumerate}
Finally,  player II wins $\mathcal{G}_\theta(\M)$ if  $\M_\alpha$ is well founded for all $\alpha<\theta$.

\begin{definition}A run of the iteration game $\mathcal{G}_\theta(\M)$ in which no player has lost after $\theta$ many steps is called an \emph{iteration tree }\index{iteration!tree $\itr$} on $\M$ of length $\theta$ and it has the form $\itr:=\{<_\itr, \M_\alpha^\itr, E_\alpha^\itr:\alpha<\theta\}$.  In this case, we write $\lh(\itr)=\theta$.  \end{definition}
\noindent If $\lh(\itr)=\gamma+1$ for some $\gamma$, then $e^\itr:=e_{0,\gamma}\to \M_\gamma^\itr$ is called the \emph{iteration embedding}\index{iteration!embedding}.  On the other hand, if $\lh(\itr)$ is a limit ordinal, $b$ is a cofinal branch through $<_\itr$ and $\alpha\in b$, then $e^\itr_{\alpha,b}:\M_\alpha^\itr\to \M_b^\itr$ is the direct limit embedding, where \[\M_b^\itr= \text{dir}\lim \{\M_\beta:\beta\in b\}.\]
\begin{definition}Let $\itr$ be an iteration tree on an $X$-premouse $\M$, with $\lh(\itr)$ a limit ordinal.  Then
\begin{enumerate}[(i)]\item \index{iteration!tree $\itr$!delx@$\delta(\itr)$}$\delta(\itr):=\sup\{\lh(E_\alpha^\itr):\alpha<\lh(\itr)\}$;\item the common part model $\M(\itr)$ of the iteration tree $\itr$\index{iteration!tree $\itr$!common part model $\M(\itr)$} is the $X$-premouse built from the extender sequence $\vec{E}=\bigcup\{\res{\vec{E}^{\M_\alpha^\itr}}{\lh(E_\alpha^\itr)}: \alpha<\lh(\itr)\}$. 
\end{enumerate}\end{definition}

\begin{definition}An $X$-premouse $\M$ is called $\theta$-iterable\index{premouse!thx@$\theta$-iterable} if player II has a winning strategy $\Sigma$ in $\mathcal{G}_\theta(\M)$. $\Sigma$ is called a $\theta$-iteration strategy for $\M$.  \end{definition}

\begin{definition}We say that a $X$-premouse $\M$ is a $X$-\emph{mouse}\index{mice} if and only if it is $\w_1+1$-iterable.\end{definition}
\begin{remark}There is an important difference between tree iterations and linear iterations of extenders in a countable $X$-premouse $\M$ in which case full iterability is equivalent to $\w_1$-iterability\footnote{See \cite[\S 6]{stn}}. 
For tree iterability in a premouse, $\w_1+1$-iterability does not imply full iterability and $\w_1$-iterability does not imply $\w_1+1$-iterability.
However, if $\M$ is countable and we are assuming $\AD$, $\w_1$ and $\w_1+1$-iterability are equivalent for countable premice.  Also, when considering countable premice we have such an equivalence under the assumption that  uniqueness of cofinal well-founded branches (i.e. the direct limit model is well-founded) through an iteration tree holds. \end{remark}

\begin{remark}If $\M$ is a countable $X$-mouse, then a $\w_1$-iteration strategy for $\M$ can be coded by a set of reals.  This is not the case for $\w_1+1$-iteration strategies on $\M$.\end{remark}


\begin{definition}Let $n\geq 1$ and let $\M$ be a $X$-premouse.  We say that $\M$ is $n$-\emph{small}\index{premouse!nx@$n$-small} if and only if whenever $E$ is an extender on the $\M$-sequence with $\cp{E}=\kappa$ we have \[\M|_\kappa\not\models ``\text{ there are  $n$ Woodin cardinals''}.\]Furthermore, we say that $\M$ is $0$-small if and only if $\M$ is an initial segment of $L[X]$.\end{definition}
\begin{definition}\label{qstru}Let $\M$ be an $X$-premouse and suppose that $\itr$ is an iteration tree of limit length on $\M$.  A $\Q$-\emph{structure}\index{qx@$\Q$-structure} for $\itr$ is a $X$-premouse $\langle J^{\vec{F}}_\beta(X);\in,X,\res{\vec{F}}{\beta}, F_\beta\rangle$ such that \begin{enumerate}[(i)]\item $\M(\itr)\li\Q$ and $\delta(\itr)$ is a cutpoint of $\Q$,
\item $\Q$ is $\w_1$-iterable above $\delta(\itr)$ and 
\item $\Q$ kills definably the Woodin property of $\delta(\itr)$\footnote{For further details see \cite[Definition 2.2.1]{ut}.}.  (When $\Q$ is a proper initial segment of a $X$-premouse, it is equivalent to say that $\beta$ is minimal with \[J_{\beta+1}^{\vec{F}}(X)\models ``\delta(\itr)\text{ is not Woodin''}.)\]\end{enumerate}If a $\Q$-structure for $\itr$ exists and is unique, we denote it by $\Q(\itr)$. 
We further define 
$\Q(b,\itr)$ for a branch $b$ as the least initial segment of the direct limit model along $b$ such that the Woodin property of $\delta(\itr)$ is destroyed definably over it.
 \end{definition}

\begin{definition}\label{qitr} Let $\M$ be a $X$-premouse.  We say that a partial iteration strategy $\Sigma$ for $\M$ is guided by $\Q$-structures if and only if
for any tree $\itr$ on $\M$ of limit length  we have: \begin{enumerate}[(i)]\item If $\Q(\itr)$ exists and $b$ is the unique branch through $\itr$ such that $\Q(b,\itr)=\Q(\itr)$,\footnote{In other words, $\Q(\itr)$ is an initial segment of the direct limit model along $b$.} then $\Sigma(\itr)=b$.  \item If no such unique branch $b$ exists, then $\Sigma(\itr)$ is undefined.\end{enumerate}   \end{definition}

\subsubsection*{The mouse operator $\sharp$}

\begin{definition}\label{defsharps}For any $n\in\w$ and $x\in{}^\w\w$, $\mn{n}{x}$ \index{mice!mx@$\mn{n}{x}$}denotes any countable, sound $\w_1$-iterable $x$-premouse which is not $n$-small but all of whose proper initial segments are $n$-small. 

\noindent Similarly, provided that an $\mn{n}{x}$ exists, $M_n(x)$\index{inner models!mx@$M_n(x)$} denotes an inner model obtained from $\mn{n}{x}$ by iterating $\toe{\mn{n}{x}}$ out of the universe.  
\end{definition}

\begin{remark} 
\label{remark on uniqueness of Mn} 
In our context, $M_n^\#(x)$ is unique. This is because we always assume that $M_n^\#(x)$ exists for all reals $x$. By combined results of Neeman and Woodin (see \cite[Theorem 1.6]{nd}), this implies that fine-structurally iterable models (i.e. iterable with truncations) of this form exists. A folklore argument then shows that $M_n^\#(x)$ is unique (see \cite[Lemma 2.41]{phildis}).
\end{remark} 

\begin{remark}\label{sharpsl}When $x=0$ in the definition above, we write $M^\sharp_n$ for $\mn{n}{0}$.  In this case, $M_n$ is the minimal canonical inner model which contains $n$ Woodin cardinals; in particular, $M_0=\LL$.  Also,  when $n=0$, we write $x^\sharp$ for $\mn{0}{x}$ and, in this case $M_0(x)=\LL[x]$ is G\"{o}del's constructible universe relative to $x$ as in Definition \ref{L}.\end{remark}

\begin{remark}\label{agree}Let $x\in{}^\w\w$ and suppose that $\mn{n}{x}$ exists.  Say, $\mn{n}{x}=\langle J^{\vec{E}}_\alpha(x);\in, \res{\vec{E}}{\alpha}, E_\alpha\rangle$.  Then, $\langle J^{\vec{E}}_\alpha(x);\in, \res{\vec{E}}{\alpha},\varnothing\rangle$ is the initial segment of $M_n(x)$ up to $\alpha$.   In particular, $\mn{n}{x}$ and $M_n(x)$ have the same reals.   
\end{remark}

Originally, the notion of $0^\sharp$ was  isolated and studied by J. Silver in his dissertation in which he derived some of the consequences of the existence of $0^\sharp$ from the existence of a Ramsey cardinal\footnote{A complete exposition of these facts can be found in \cite[\S9]{kanamori}. }.  Briefly, we mention some of the equivalent definitions for $x^\sharp$, for $x\in{}^\w\w$ which will be relevant for the following chapters. 
\begin{lemma}\label{sheq}Let $x\in{}^\w\w$. The following are equivalent:
\begin{enumerate}[(1)]
\item $x^\sharp$ exists;\index{mice!xx@$x^\sharp$}
\item there is a club proper class $I$ of indiscernibles for $L[x]$;
\item there is some limit ordinal $\lambda$ such that $L_\lambda[x]$ has an uncountable set of indiscernibles;
\item there is a elementary embedding $j\colon L_\alpha[x]\to L_\beta[x]$ where $\alpha$ and $\beta$ are limit ordinals and $\cp{j}<|\alpha|$; 
\item there exists a nontrivial elementary embedding $j\colon L[x]\to L[x]$.
\end{enumerate}\end{lemma} 
\begin{proof} 
See \cite[Chapter 2, \S 9 and Chapter 4, \S 21]{kanamori} and \cite[Chapter 18]{jech}. 
\end{proof} 


\begin{lemma}Let $n\geq 1$ and suppose that for every $x\in{}^\w\w$, $\mn{n}{x}$ exists.   Let $\M$ be a countable $n$-small premouse which is $\w_1$-iterable and let $\itr$ be an iteration tree on $\M$.  Then, the $\Q$-structure iteration strategy $\Sigma(\itr)$ is the unique $\w_1$-iteration strategy on $\M$.  \end{lemma}
\begin{proof}See \cite[Corollary 6, \S6]{jen9} or  \cite[Corollary 6.14]{steel}.\end{proof}



One important feature of having Woodin cardinals is the increasing amount of correctness between $\mn{n}{x}$ and the universe $\VV$.  In the next theorem, we use the term $n$-\emph{iterability} in the sense of \cite[Definition 1.1]{nd}. 
It is defined as $\omega\cdot n$-iterability (without truncations, as above) with a uniqueness condition for branches.  

\begin{lemma}[Woodin]\index{correctness!Woodin's Theorem} \label{corr} Let $n<\w$ and $x\in{}^\w\w$.  Suppose that $\N$ is a proper class $x$-premouse which is $n$-iterable and has $n$ Woodin cardinals which are countable in $V$.\:  Let $\varphi$ be a $\Sigma^1_{n+2}$-formula and let $a\in \N\cap{}^\w\w$.\:  Then: 
\begin{enumerate}[(1)]
\item If $n$ is even,\:  $\varphi(a)\iff \N\models \varphi(a)$.
\item If $n$ is odd and $\delta_0$ is the least Woodin cardinal in $\N$, we have\[\varphi(a)\iff \forc{\N}{\text{Col}(\w, \delta_0)} \varphi(a)\]Furthermore, in this case we have that $\N$ is $\bsi{n+1}$-correct in $\VV$.\end{enumerate} 
\end{lemma}
\noindent For a proof of this result see \cite[Lemma 1.3.4]{ut}.    
\begin{remark}Suppose that $\N$ is a $x$-premouse as in Lemma \ref{corr}.  Then, from (1) we conclude that for even $n$, $\N\prec_{\bsi{n+2}}\hspace{-2mm}V$.  If we additionally assume that in $\N$, there is a measurable cardinal above the $n$ Woodin cardinals (and $\N$ is still a proper class), then $\N\prec_{\bsi{n+2}}\hspace{-2mm}V$ also holds for odd $n$. \end{remark}

\begin{remark} 
It follows that for any transitive model $M$ of $\ZF$ such that  for all reals $x\in M$, both $M_n^\#(x)^M$ and $M_n^\#(x)^V$ exist and are equal, we have $M\prec_{\Sigma^1_{2n+2}}V$. 
\end{remark}

\rm{There is a seminal result between inner model theory and descriptive set theory, linking projective determinacy level-by-level with the existence of inner models with Woodin cardinals. } 

\begin{theorem}[Martin, Harrington, Neeman, Woodin]\label{micedet}Let $n\in\w$.  Then, the following are equivalent:
\begin{enumerate}[(1)]\item For every $x\in{}^\w\w$, $\mn{n}{x}$ exists and is $\w_1$-iterable.
\item Every $\boldsymbol{\Pi}^1_{n+1}$ set is determined.\end{enumerate} \end{theorem}
The case $n=0$ of theorem \ref{micedet} is due to D.A. Martin and L. Harrington (cf. \cite{md} and \cite{had}) and the forward direction is due to Itay Neeman \cite{nd} for $n$ even and to Hugh Woodin for $n$ odd.  The backward direction is a result announced by H. Woodin in the 80's, but until recently without a published proof. For a proof of this result see \cite{ut}.  

\section{Tree forcings and capturing}\label{sharpsreals}
 
\begin{definition}[{\cite[Definitions 4.8 and 4.13]{steelm}}]A \emph{mouse operator}\index{mouse operator} $\M$ is a function $x\mapsto \M(x)$ with $\dom{\M}=\HC{1}$ such that there is a sentence $\varphi$ in the language of relativized premice so that for every $x\in\HC{1}$:\begin{enumerate}[(i)]\item $\M(x)$ is a $\w$-sound, $\w_1+1$-iterable $x$-premouse,
\item $\M(x)\models \varphi$ and
\item for all $\xi\in\OR\cap\M(x)$, $\M(x)|_\xi\not\models\:  \varphi$.\end{enumerate}\end{definition}
Recall that if $\PP$ is a strongly arboreal forcing notion and $G$ is $\PP$-generic over $V$, then $G$ and $x_G=\bigcap\{[T]: T\in G\}$ are inter-definable, i.e. $V[G]=V[x_G]$ (see \ref{arb}).
\begin{definition}\label{pqprop}Let $\PP$ and $\QQ$ be strongly arboreal forcing notions in $\VV$.  Let $\M:x\mapsto \M(x)$ be a mouse operator defined on $\VV\cap\HC{1}$.  We say that $\QQ$ \emph{captures} $\PP$\index{capturing} over $\M$  if whenever $G$ is a $\PP$-generic filter over $V$, for $\leq_L$-cofinally many $a\in \RR\cap V[G]$ there is $x\in\RR\cap V$ such that the following holds:
\begin{enumerate}[(i)]
\item 
$\M(x)[x_G]$ is a $\QQ$-generic extension of $\M(x)$, 
\item $G\cap\M(x)[x_G]\in\M(x)[x_G]$ and 
\item $a\in\M(x)[x_G]$. 
\end{enumerate}
\end{definition}

We will see a number of examples to the definition above assuming closure under the sharp operators $\M^\sharp_n: x\mapsto \mn{n}{x}$ in the following sections. 

The next remark was suggested by the anonymous referee. 

\begin{remark} 
Any $\Sigma^1_2$ provably c.c.c. forcing $\PP$ captures itself over $\M^\sharp_n$, assuming $\M^\sharp_n$ is total on $\VV\cap\HC{1}$. 
To see this, note that $M^\sharp_n(x)$ is preserved by $\PP$ for all reals $x$  by \cite[Lemma 3.11]{phildis}. 
If $G$ is $\PP$-generic over $V$, $a\in \RR\cap V[G]$ and $\tau$ a nice name with $\tau^G=a$, then it is easy to see that $G\cap M^\sharp_n(\tau)$ is $\PP$-generic over $M^\sharp_n(\tau)$ and $a\in M^\sharp_n(\tau)[G\cap M^\sharp_n(\tau)]$. 
\end{remark}

\subsection{$\CC$-capturing of Sacks and Silver forcing}\label{sacksc}

Let $\CC={}^{<\omega}2$ denote Cohen forcing.  \: We aim to show that Cohen forcing captures Sacks and Silver forcing over the operator $\M_n^\sharp$.  First, we define auxiliary forcings giving us perfect trees whose branches are Cohen reals. 

\subsubsection{Auxiliary construction for Sacks forcing}
\begin{definition}\label{amos}Suppose that $S\subset {}^{<\w}2$ is a perfect tree.  We define:\index{forcing notions!aax@$\Amo{\sa}{S}$}
\[\Amo{\sa}{S}=\{t\subset S: t\text{\: is a finite subtree of  $S$ such that $\Spl(t)$ is isomorphic to some $^n2$ }\}\]
ordered by end-extension, i.e. $t\leq s$ if and only if $t\supset s$ and $\Spl(t)\cap s\subseteq \Spl(s)\cup \Ter(s)$. 


\noindent If $S\in\sa$,\: let $\Theta_S$ the natural isomorphism between $\Spl(S)$ and ${}^{<\w}2$.

\end{definition}

\begin{proposition}\label{denseperfect}Let $S\in\sa$ and let $D$ be a dense subset in $\CC$.\:  Then, the set \[D_{S}=\{t\in\Amo{\sa}{S}: \forall s\in\Ter(t)(\Theta_S(s)\in D)\}\] is dense in $\Amo{\sa}{S}$.\end{proposition}
\begin{proof}Let $t\in\Amo{\sa}{S}$ and let $s\in\Ter(t)$.\:  As $D$ is dense in $\CC$, there is $s^D\in D$ such that $\Theta_S(s)\subseteq s^D$.\:  We can define the tree $t'=\text{Downcl}_{S}(t\cup\{\Theta_S^{-1}(s^D):s\in\Ter(t)\})$.\:  It is clear that $t'\supset t$ and $\Ter(t')=\{\Theta_S^{-1}(s^D):s\in \Ter(t)\}$, therefore $t'\in D_{S}$ and $t'\leq t$ as required. \end{proof}

\begin{lemma}\label{perfect}Let $S\in\sa$ and suppose that $g$ is $\Amo{\sa}{S}$-generic over a transitive model $M$.\:  Then:
\begin{enumerate}[(i)]
\item $T_g=\bigcup g$ is a perfect subtree of $S$.
\item For every $x\in[T_g]$, $\Theta_S(x):=\bigcup_{n<\w}\Theta_{S}(\res{x}{n})$  is Cohen-generic over $M$.\end{enumerate}
\end{lemma}

\begin{proof}Since $\Theta_S:\Spl(S)\to {}^{<\w}2$ is an isomorphism it is enough to prove the lemma for  $\amos= \Amo{\sa}{{}^{<\w}2}$.\:   For a tree $t\in \amos$, let $\Ter(t)=\{s\in t: \succc{t}{s}=\varnothing\}$.\:  \: 

  \begin{enumerate}[(i)]
\item  For each $n<\w$, the set $D^n=\{t\in \amos: \heigth{t}>n\}$  is open dense in $\amos$.\:   If $p\in T_g$, consider  the finite subtree $t=\res{T_g}{p}\in \amos$.  If $n=\heigth{t}$, take $s\in D^{n+2}\cap G$ extending $t$.   Then, $p\in s\smallsetminus \Ter(s)$ and, as $s$ is perfect, $p^\frown 0$, $p^\frown 1\in s\subseteq T_g$.     Therefore $T_g$ is an perfect tree.

\item Let $x\in[T_g]$ and let $D\in M$ be dense in Cohen forcing.\:  By \ref{denseperfect}, $D_{{}^{<\w}2}$ is dense in $\amos$, so there is $t\in D_{{}^{<\w}2}\cap G$.\:  Since $\res{T_g}{\heigth{t}}=t$ we have that $\res{x}{\heigth{t}}\in \Ter(t)\subseteq D$.\:  Therefore, $x$ is $\CC$-generic over $M$.\qedhere\end{enumerate}\end{proof}

\subsubsection{Auxiliary construction for Silver forcing}

\begin{definition}\label{amou}\index{forcing notions!abx@$\Amo{\UU}{T}$}Suppose that $T\subseteq {}^\w2$ is an uniform tree. \:  Let\[\Amo{\UU}{T}=\{t\subseteq T:\text{ $t$ is a finite uniform subtree of $T$}\}.\]
If $t\in\Amo{\UU}{T}$ and $n<\w$, let $\res{t}{n}=\{p\in t: \lh(p)\leq n\}$.  Since $t\in\Amo{\UU}{T}$ is finite and uniform, we can define $\heigth{t}$ as being the height of any branch through $t$; let also $\Ter(t)$ be the set of terminal nodes of $t$.  Finally, in $\Amo{\UU}{T}$ we stipulate $s\leq t$ if and only if $\heigth{s}\geq \heigth{t}$ and $\res{s}{\heigth{t}}=t$.\end{definition}

Since the poset $(\Amo{\UU}{T},\leq)$ is countable  and atomless, this forcing notion is equivalent to Cohen forcing.\:  

\begin{proposition}\label{densesilver}Let $T\in \UU$ and $D\subseteq \CC$ open dense.  Then, the set 
\[D_T=\{t\in \Amo{\UU}{T}\mid \forall s\in \Ter(T) (\Theta_T(s)\in D) \} \] 
is dense in $\Amo{\UU}{T}$, where $\Theta_T: T\to {}^{<\w}2$ is the natural isomorphism between the splitting points of $T$ and the full binary tree. \end{proposition}
\begin{proof} For $a,b\in\CC$, if $m=|b|>|a|$ we define the sequence $a\oplus b\in{}^{m}2$ as follows:
\[a\oplus b (n)=\begin{cases} a(n) & \mbox{if  } n<|a|,\\
b(n)& \mbox{if  }|a|\leq n<m.\end{cases}\]    Let $t\in\Amo{\UU}{T}$ and suppose $\Ter(t)= \{t_i:i< m\}$.\:  We will construct inductively the terminal nodes of an extension for $t$ in $\Amo{\UU}{T}$:
\renewcommand{\labelitemi}{\checkmark}
\begin{enumerate}[(i)]\item   As $D$ is dense in $\CC$, pick $s_0\in D$ extending $\Theta_T(t_0)$.\:  
\item For $i<m$, pick $s_i\in D$  such that $ \Theta_T(t_i)\oplus s_{i-1}\subseteq s_i$.\:  
\item Finally, for each $i<m$ let $t_i^*:=\Theta_T(t_i)\oplus s_{m-1}$.
\end{enumerate}

Note that for each $i<m$, $\lh(t_i^*)=\lh(t_j^*)$.  Let $E=\{\Theta_T^{-1}(t_i^*): i<m\}$ and define $t'=\mathrm{dowcl}_{T}(E)$.\:  By construction, we have that $t'\in\Amo{\UU}{T}$ and $t'\leq t$.  Also,  since $E=\Ter(t')$ we have  $\Ter(\Theta_T(t'))=\Theta_T[\Ter(t')]=\Theta_T[E]=\{t_i^*: i<m\}\subseteq D$.  
\end{proof}

\begin{lemma}\label{uniform} Let $T\in \UU\cap M$, where $M$ is a transitive model of $\ZFC$, and suppose that $g$ is $\Amo{\UU}{T}$-generic over $M$.\:  Then, $T_g=\bigcup g$ is a Silver subtree of $T$ whose branches  are $\CC$-generic over $M$ modulo $\Theta_T$. \end{lemma}

\begin{proof}We show that the branches of $T_g$ are $\CC$-generic over $M$ modulo the isomorphism $\Theta_T: \Spl(T)\to {}^\w2$.\:   Let $x\in[T_g]$ and let $D\in M$ be a dense subset of $\CC$.\:  As $D_T$ is dense in $\Amo{\UU}{T}$, pick $t\in D_T\cap g$.\:  If $\heigth{t}=n$ then $\Ter(\Theta_T(\res{x}{n}))\subseteq D$.\:  Therefore,\: $x$ is $\CC$-generic over $M$ modulo the isomorphism $\Theta_T$.   
\end{proof}

Now we have all the necessary tools to prove our result.

\begin{lemma}\label{dense}\index{capturing!between!cax@$\CC$ and $\sa$}\index{capturing!between!cbx@$\CC$ and $\UU$}Let $n\in\w$ and suppose that $\VV\cap\HC{1}$ is closed under the sharp operator $x\mapsto \mn{n}{x}$.\: Then, $\CC$ captures Sacks and Silver forcing over $\M_n^\sharp$. 
 \end{lemma}
 
 \begin{proof}
Let $G$ be $\PP$-generic over $V$ where $\PP$ is Sacks or Silver forcing and let $r\in{}^\w\w\cap V[G]$.   We want to see that there is some $x\in {}^\w\w\cap V$ such that $r\in\mn{n}{x}[x_G]$. 

Let $\tau$ be a $\PP$-name for $r$.   By \ref{cnames}, there is a condition $P\in G$ and a $\PP$-name $\sigma\in \HC{1}$ with $P\forc{\PP}{V}\tau=\sigma$.  
\begin{claim}\label{key1}The set\[D_P=\{T\in\PP: \exists S\: ( P\geq S\geq T \wedge \forall z\in[T] (\text{$z$ is $\CC$-generic over $\mn{n}{\sigma, S}$ mod $\Theta_S$})) \}\footnote{While one can let $S=P$ here,  the present formulation is closer to the proofs of Lemmas \ref{densemath1}, \ref{densesm} and \ref{la-mi} below. $S$ is needed there because a $\PP$-generic real is not necessarily $\PP$-generic over inner models.} \] 
is $\PP$-dense below $P$.\end{claim}  
\begin{proof} Let $S\leq P$.\:   Since $\mn{n}{\sigma, S}$ is countable, we can pick a $\Amo{\PP}{S}$-generic $T\in\VV$ over $\mn{n}{\sigma, S}$.\:   By  \ref{perfect} in the case of $\PP=\sa$ or \ref{uniform} if $\PP=\UU$, we have that $T\leq S$ and in $\VV$ holds that every branch through $T$ is $\CC$-generic over $\mn{n}{\sigma, S}$ modulo $\Theta_S$.  Therefore, $D_P$ is dense below $P$.\end{proof}

Let $x_G$ the $\PP$-generic real associated to $G\ni P$.  By Claim \ref{key1}, there is some $T\in D_P\cap G$.   Hence, there is some $S\geq T$ such that \[V[G]\models ``x_G \text{\: is\:  $\CC$-generic over\: }\mn{n}{\sigma,S}\text{ modulo $\Theta_S$''}.\]
Finally, since \[G\cap\mn{n}{\sigma, S}[x_G] =\{T\in \PP\cap \mn{n}{\sigma, S}: x_G\in[T]\}\in \mn{n}{\sigma,S}[x_G],\] 
it follows that $r=\sigma^G\in \mn{n}{\sigma,S}[x_G]$.
\end{proof}

\subsection{$\MM$-capturing of forcings in $\T$}

\rm {Since Mathias forcing does not add any Cohen reals\footnote{For a proof of this fact, see  \cite[Sect. 7.4.A]{judah}.} we cannot expect to construct an auxiliary forcing for $\MM$ like in \ref{amos} and \ref{amou}.    

\noindent Nevertheless we found an analogous result to Lemma \ref{dense} without the use of Cohen forcing.   For this, we will make use of the Mathias property established in \ref{mathiasproperty}.} 

\begin{lemma}\label{densemath1} Let $n\in\w$ and suppose that $\VV\cap\HC{1}$ is closed under the operator $\M_n^\sharp:x\mapsto \mn{n}{x}$.  Then, Mathias forcing captures itself over  $\M_n^\sharp$.\index{capturing!between!max@$\MM$ and itself}\end{lemma}
 
 \begin{proof}Let $G$ be $\MM$-generic over $V$ and let $r\in{}^\w\w\cap \VV[G]$.  Given a $\MM$-name $\tau$ for $r$, by properness, there exists a condition $p\in G $ and a countable $\MM$-name $\sigma$ such that $p\forc{\MM}{\VV}\tau=\sigma$. 
 
For $q=(s,A)\in \MM$, write $[(s,A)]_\MM:=\{X \in[\omega]^\omega \mid s\subseteq X \subseteq s\cup A\}$.

\begin{claim}The set \[D_p=\{q:\exists p' (p\geq p'\geq q \wedge \forall {z} \in [q]_\MM\: (z \text{ is $\MM$-generic over $\mn{n}{\sigma, p'}$}))\}\footnote{Recall that for a Mathias condition $q=(t,B)$, $[q]_\MM=\{z\in [\w]^\w: t\subset z\subset t\cup B\}$.}\] is $\MM$-dense below $p$.\end{claim} 
\begin{proof}  Let $p'\leq p$.  Since  $\mn{n}{\sigma, p'}$ exists in $\VV$ and it is countable, we can pick a $\MM$-generic filter $H\in\VV$ over $\mn{n}{\sigma,p'}$, containing the condition $p':=(s, A)$.\:   In this case, the Mathias real $x_H$ satisfies $s\subset x_H\subset s\cup A$. \:  Let $q=(s, x_H\smallsetminus s)\in\MM$.\:  Since $x_H\smallsetminus s\subset A$ we have $q\leq p'$. \:  Now, if $z\in [q]_\MM$ we have that $s\subset z\subset s\cup x_H\smallsetminus s = x_H$.  Thus, by Lemma \ref{mathiasproperty} we have that $z$ is also $\MM$-generic over $\mn{n}{\sigma,p}$.  Thus, $q\in D_p$.
\end{proof}
Since $G\ni p$, by the claim above there is some $q\in D_p\cap G$. If $p'\leq p$ is a condition witnessing that $q\in D_p$, as $x_G\in[q]_\MM$ we have\[V[G]\models \text{``$x_G$ is $\MM$-generic over $\mn{n}{\sigma,p'}$'' }.\]
\noindent Finally, since $G=\{q\in\MM: x_G\in [q]_\MM\}$ 
is computable inside $\mn{n}{\sigma,p'}[x_G]$ we have that $G\cap\mn{n}{\sigma,p'}[x_G]\in \mn{n}{\sigma,p'}[x_G]$.  Therefore, since $r=\sigma^G$, we have that $r\in\mn{n}{\sigma,p'}[x_G]$.
 \end{proof}

\subsubsection{Sacks and Silver reals viewed as local Mathias reals}

 \begin{lemma}\label{densesm} Let $n<\w$ and suppose that for every $x\in{}^\w\w$, $\mn{n}{x}$ exists.  Then, $\MM$ captures Sacks and Silver forcing over the mouse operator $\M^\sharp_n$.\index{capturing!between!sax@$\MM$ and $\sa$}\index{capturing!between!silx@$\MM$ and $\UU$}\end{lemma}
 
 \begin{proof}We give the argument for Sacks forcing. \: Throughout this proof, we identify $\MM$ with its tree version $\MM'_{\text{tree}}$ (see \ref{mathiastree}).  Assume that $G$ is $\sa$-generic over $V$ and  let $r\in{}^\w\w\cap V[G]$.  Then, given  $\tau$ a $\sa$-name for $r$, there is a condition $P\in\sa$ and some $\sigma\in\HC{1}$  such that $P\forc{\sa}{\VV} \sigma=\tau$.  

\begin{claim}If $\dot{x}$ is a $\sa$-name for a Sacks generic real, let \[D_P=  \{T: \exists S (P\geq_\sa S\geq_\sa T\wedge (T\forc{\sa}{V} \dot{x} \text{ is $\MM_{\text{tree}}$-generic over $\mn{n}{\sigma, S}$ modulo $\Theta_S$}))\}.\]  Then, $D_P$ is $\sa$-dense below $P$.\end{claim}

\begin{proof} Let  $S\leq_\sa P$ and let $\Theta_S: \Spl(S)\to {}^{<\w}2$ be the canonical isomorphism.\:  Since $\mn{n}{\tau,S}$ exists and it is countable,  we can pick an $\MM$-generic real $u\in\VV$ over $\mn{n}{\tau,S}$.\: Let $\bar{T}=U_{(\varnothing, u)}\in\MM'_{\text{tree}}$ be the subtree of ${}^{<\w}2$ determined by the characteristic function of $u$ as in \ref{mathiastree}.   Finally let $T=\text{downcl}_S(\Theta_S^{-1}[\bar{T}])$.\:  Observe that $T\in\sa$ and $T\leq_\sa S\leq_\sa P$.  Suppose that $x$ is Sacks generic over $V$ below the condition $T$.  Then, take \[\bar{x}=\bigcup\{\Theta_S(\res{x}{n}):n\in\w,\: \res{x}{n}\in\Spl(S)\}.\]
By construction, $\bar{x}\in [\bar{T}]$, hence $v=\{n\in\w: \bar{x}(n)=1\}\subset u$.  Note that $v$ is infinite.  Otherwise, $\bar{x}\in \VV$ and therefore $x\in \VV$ contradicting the assumption that $x$ is $\sa$-generic over $\VV$.   
Thus, by fact \ref{mathiasproperty}, $v$ is $\MM$-generic over $\mn{n}{\tau, S}$, therefore $\bar{x}$ is $\MM'_{\text{tree}}$-generic over $\mn{n}{\tau, S}$. 
Since $x$ and $\bar{x}$ are interdefinable via $\Theta_S$, we conclude that \[T\forc{\sa}{\VV} \dot{x}\text{\:  is $\MM'_{\text{tree}}$-generic over }\mn{n}{\tau,S}\text{ modulo $\Theta_S$}.\qedhere\]
\end{proof}
Let $T\in D_P\cap G$ and let $S\in\sa$ witnessing $T\in D_P$.  Thus \[V[G]\models x_G\text{ is $\MM'_{\text{tree}}$-generic over $\mn{n}{\sigma,S}$ modulo $\Theta_S$}.\]In \ref{dense}, we saw already that $G\cap\mn{n}{\sigma,S}[x_G]\in\mn{n}{\sigma,S}[x_G]$.  In particular, since $r=\sigma^G$ it follows that $r\in\mn{n}{\sigma,S}[x_G]$.
 \end{proof}
  
\subsubsection{Laver and Miller reals viewed as local Mathias reals}

We will use the following fact: 

\begin{fact}\label{codeslaver}  Suppose that  $u\subseteq\omega$ is an infinite set.\:   Then,  $T_{(\varnothing,u)}=\{t\in{}^{\uparrow<\w}\w: \ran{t}\subset u\}\in\MM_{\text{tree}}$ is a Laver tree with stem $\varnothing$.\end{fact} 




 \begin{lemma} \label{la-mi}

 Let $n<\w$ and suppose that for every $x\in{}^\w\w$, $\mn{n}{x}$ exists.  Then $\MM$ captures Laver forcing over the mouse operator $\M_n^\sharp$.\index{capturing!between!lavx@$\MM$ and $\LLL$}
\end{lemma} 

\begin{proof}Throughout this proof, we shall identify $\MM$ with its tree version $\MM_{\text{tree}}$.  Let $G$ be $\LLL$-generic and  let $r\in{}^\w\w\cap V[G]$.  Given a $\LLL$-name $\tau$ for $r$, let $P\in G$ such that  such that $P\forc{\LLL}{\VV} \tau=\sigma$ for some $\sigma\in\HC{1}$.
\begin{claim}If $\dot{x}$ is an $\LLL$-name for a Laver real, let \[D_P=  \{T: \exists S (P\geq_\LLL S\geq_\LLL T\wedge (T\forc{\sa}{V} \dot{x} \text{ is $\MM_{\text{tree}}$-generic over $\mn{n}{\sigma, S}$ modulo $\Theta_S$}))\}.\]  Then $D_P$ is $\LLL$-dense below $P$.\end{claim}
\begin{proof}Let $S\leq P$.   Since $\mn{n}{\sigma,S}$ exists and it is countable,  we can pick $u\in\VV$ a $\MM_{\text{tree}}$-generic real over $\mn{n}{\sigma,S}$.\:  By identifying $u\in{}^{\uparrow\w}{\w}$ with its range we have, by \ref{codeslaver}, that $T_{(\varnothing,u)}\in\VV$ is a Laver tree with $\stem(T_{(\varnothing,u)})=\varnothing$.\:  
Therefore, if $\Theta_S: \Spl(S)\to {}^{\uparrow<\w}\w$ denotes the canonical isomorphism between $S$ and ${}^{\uparrow<\w}\w$ we have that $T=\text{downcl}(\Theta_S^{-1}(T_{(\varnothing,u)}))$ is a Laver subtree of $S$.

\noindent Suppose $x$ is a $\mathbb{L}$-generic real over $\VV$ below $T$.  Let \[\bar{x}=\bigcup_{i\geq |\stem(S)|} \Theta_S(\res{x}{i}).\] 
Observe that $\bar{x}$ is a branch through the tree $T_{(\varnothing,u)}$.  Therefore, $\ran{\bar{x}}\subset u$.  Since $x$ is infinite so is $\ran{\bar{x}}$.  By Lemma 
 \ref{mathiasproperty}, we have that $\bar{x}$ is an $\MM_{\text{tree}}$-generic real over $\mn{n}{\sigma, S}$.\:  Since $\bar{x}$ and $x$ are interdefinable via $\Theta_S$ we conclude \[T\forc{\LLL}{\VV} \dot{x} \text{ is }\MM_{\text{tree}}\text{-generic over $\mn{n}{\sigma, S}$ mod $\Theta_S$}.\qedhere\]
\end{proof}As before, take $T\in D_P\cap G$ and let $S\in\LLL$ be  a witness for that.  Then we have \[V[G]\models x_G\text{ is $\MM_{\text{tree}}$-generic over $\mn{n}{\sigma,S}$ modulo $\Theta_S$}.\]As in the former cases, since $G= \{T\in\LLL: x_G\in[T]\}$ is computable in $\mn{n}{\sigma,S}[x_G]$ we have that $G\cap\mn{n}{\sigma,S}[x_G]\in\mn{n}{\sigma,S}[x_G]$.  Thus,  $r= \sigma^G\in\mn{n}{\sigma,S}[x_G]$.
\end{proof}

\begin{corollary}Let $n<\w$ and suppose that for every $x\in{}^\w\w$, $\mn{n}{x}$ exists.  Then, $\MM$ captures $\ML$ over the mouse operator $\M_n^\sharp$\index{capturing!between!milx@$\MM$ and $\ML$}.\end{corollary}

\begin{proof}By \ref{codeslaver}, if $u\subset\omega$ is  infinite, then $T_{(\varnothing,u)}$ is a Laver tree.   In particular, $T_{(\varnothing,u)}$ is a Miller tree with stem $\varnothing$.  With this remark, the proof follows same as before, just replacing $\LLL$ by $\ML$.   \end{proof}

Summarizing the lemmas given through this section, we can state the following key  result:

\begin{theorem}\label{keylemma}Let $n<\w$ and let $\PP\in\T$.   Assume that for every $ x\in{}^\w\w$, $\mn{n}{x}$ exists and is $\w_1$-iterable.   Then $\MM$ captures $\PP$ over the mouse operator $\M_n^\sharp$.  In particular, given a $\PP$-generic filter $G$ over $V$, for every real $r\in V[G]$ there is a real $y\in V$ such that \[V[G]\models ``x_G\text{ is an $\MM$-generic real over $\mn{n}{y}$ mod $\Theta$ and $r\in\mn{n}{y}[x_G]$''},\]where $\Theta$ is a definable function in $y$.  Furthermore, in this situation also \[V[G]\models ``x_G\text{ is an $\MM$-generic real over $M_n(y)$ mod $\Theta$ and $r\in M_n(y)[x_G]$''},\]where $M_n(y)$ is as in Definition \ref{defsharps}. \end{theorem}

\begin{proof}For the second part, note that for  every $\PP$-name $\sigma\in\HC{1}$ for a real $r\in V[G]$ and for $p\in\PP$, we have $\RR\cap \mn{n}{\sigma,p}= \RR\cap M_n(\sigma,p)$ by Remark \ref{agree}. Therefore, we can pick some real $u\in\VV$ which is $\MM$-generic over $M_n(\sigma,p)$.  By proceeding  as in the proof of \ref{densemath1}, \ref{densesm} or \ref{la-mi}, we conclude the desired result.\end{proof}

\subsection{Addendum: Capturing products of forcings in $\T$}

\begin{lemma}
\label{keylemma for product forcings} 
Let $n\in\w$ and suppose that for all $x\in{}^\w\w$, $\mn{n}{x}$ exists and is $\w_1$-iterable.  For each $\PP\in\T$,  $\MM_2=\MM\times\MM$\index{capturing!between!wx@$\MM\times\MM$ and forcings in $\T$} captures $\PP_2=\PP\times\PP$ over $\M_n^\sharp$.\end{lemma}
\begin{proof}We prove the result for Mathias forcing.  The other arguments are analogous, by using the tree versions of Mathias forcing $\MM'_{\text{tree}}\subset{}^{<\w}2$ in the case of Sacks and Silver forcing and $\MM_{\text{tree}}\subset{}^{\uparrow<\w}\w$ for Laver and Miller forcing.   

\noindent Let $G$ be $\MM_2$-generic over $V$ and suppose $r\in{}^\w\w\cap\VV[G]$.    Since $\MM_2$ is proper, if $\tau$ is a $\MM_2$-name for $r$, there is $(p,q)\in G$ and a $\MM_2$-name $\sigma\in\HC{1}$ such that $(p,q)\Vdash \sigma=\tau$. \:   Let $(p_0,q_0)\leq (p,q)$.   Since $\mn{n}{p_0,q_0,\sigma}$ exists in $\VV$ and it is countable, we can pick $H\in\VV$ containing $(p_0,q_0)$ which is $\MM_2$-generic over $\mn{n}{p_0,q_0,\sigma}$.   

 Let $x_H=\bigcup \{\stem(p^*):p^*\in H_0\}$ and $y_H=\bigcup \{\stem(q^*):q^*\in H_1\}$, where $H_0$ and $H_1$ are the first and second projections of $H\subset \MM\times\MM$.  Then, $x_H$ is $\MM$-generic over $\mn{n}{p_0,q_0,\sigma}$ and $y_H$ is $\MM$-generic over $\mn{n}{p_0,q_0,\sigma}[x_H]$.  Suppose that $p_0=(s,A)$, $q_0=(t,B)\in \MM$.  Then, $s\subset x_H\subset s\cup A$ and $t\subset y_H\subset t\cup B$; also, as $H\in\VV$ we have that $x_H$ and $y_H$ are in $\VV$.  Let $p_1=(s, x_H\setminus s)$ and $q_1=(t, y_H\setminus t)$.  Notice that $(p_1,q_1)\leq_\MM(p_0,q_0)$.   

Suppose that $x_0\in[p_1]_\MM$ and $y_0\in[q_1]_\MM$.  Then, $s\subset x_0\subset (x_H\setminus s)\subset x_H$ and  $t\subset y_0\subset (y_H\setminus s)\subset y_H$.  Since $x_0$ and $y_0$ are infinite, we have that $x_0$ and $y_0$ are also $\MM$-generic over $\mn{n}{p_0,q_0,\sigma}$ and $\mn{n}{p_0,q_0,\sigma}[x_H]$ respectively.      Thus, we have proved that the set \[D=\{(p_1,q_1):\exists (p_0,q_0)\geq (p_1,q_1)(\forall x_0\in[p_1], y_0\in[q_1] ((x_0,y_0)\text{ is $\MM_2$-generic over $\mn{n}{p_0,q_0,\sigma}))\}$}\]is $\MM_2$-dense below $(p,q)$. 

\noindent For $G$, let $G_0$ and $G_1$ its projections as before.  If $x_G$ is the Mathias real over $\VV$ associated to $G_0$ and $y_G$ is the Mathias real over $\VV[x_G]$ associated to $G_1$ and $(p_1,q_1)\in D\cap G$ with$(p_0,q_0)$ as a witness, we have  \[ \VV[G]\models (x_G,y_G)\text{ is $\MM^2$-generic over $\mn{n}{p_0,q_0,\sigma}$}.\]Since the generic filter $G$ is computable in $\mn{n}{p_0,q_0,\sigma}[(x_G,y_G)]$, it follows that condition (ii) in Definition \ref{pqprop} is also satisfied.\end{proof}

\subsubsection{Capturing $\sa_2$ and $\UU_2$ by Cohen forcing }
\begin{definition}\label{sacksamo} If $(T,S)\in\sa_2=\sa\times\sa$, we define
\[\mathbb{A}_{\sa_2, (T,S)}=\mathbb{A}_{\sa, T}\times\mathbb{A}_{\sa, S}\]
ordered componentwise. \end{definition}
Recall that for a condition $T\in\sa$, $\Theta_T:\Spl(T)\to {}^{<\w}2$ denotes the canonical isomorphism.

\begin{proposition}\label{denseproduct}  Let $T, S\in \sa$ and let $D\subset \mathbb{C}_2=\CC\times\CC$ be an open dense set.\:  Then, the set
\[D_{(T,S)}=\{(t, s)\in\mathbb{A}_{\sa_2, (T,S)}:\forall (p,q)\in\Ter(t)\times\Ter(s) (\Theta(p,q)\in D) \}\]is a dense subset of $\mathbb{A}_{\sa_2, (T,S)}$, where $\Theta:=\Theta_T\times\Theta_S$.\end{proposition}

\begin{proof}Given $(t,s)\in\mathbb{A}_{\sa_2, (T,S)}$ suppose that $\Ter(t)=\{p_i^t: i<n\}$ and $\Ter(s)=\{p_i^s: i<m\}$.\:  Inductively, consider the pair $(p_0^t, p_0^s)$.\:  As $D$ is dense in $\mathbb{C}^2$, take ${(p_0^t, p_0^s)}^*$ extending $(p_0^t, p_0^s)$ such that $\Theta({(p_0^t, p_0^s)}^*)\in D$.\:  Since $D$ is open dense and $\Theta$ is an order isomorphism, any further extension of this pair is also in $D$.   Now, consider $((p_0^t)^*, p_1^s)$ and find ${((p_0^t)^*,p_1^s)}^{*}$ such that  $\Theta({((p_0^t)^*,p_1^s)}^{*})\in D$.  At stage $(0, m-1)$ we have found extensions in $D$ (modulo $\Theta$) for the first $\leq_{\text{lex}}$-$m$ pairs $(p_0^t, p_i^t)$, $i<m$.

At stage $(i,j)$, we find an extension in $D$ modulo $\Theta$ of  the current extension of $(p_i^t,p_j^s)$.\: At the end of the construction we have found for each $(i,j)$, $(q_i^t, q_j^s)\geq (p_i^t, q_j^s)$ such that $\Theta((q_i^t, q_j^s))\in D$.  

Let $t'= \text{Downcl}_T(t\cup\{q_i^t:i<n\})$, $s'= \text{Downcl}_S(s\cup\{q_i^s:i<m\})$.  Note that $(t,s)\leq(t',s')$ and $(t', s')\in D_{(T,S)}$ as required. \end{proof}

\begin{proposition}Suppose $(T,S)\in\sa_2$ and let $G$ be $\mathbb{A}_{\sa_2,(T,S)}$-generic over $V$.\:  Then: 
\begin{enumerate}[(1)]\item $(T_G, S_G)=\bigcup G$ is such that $T_G\leq_\sa T$, $S_G\leq_\sa S$.
\item For every $(x,y)\in [T_G]\times [S_G]$, $(x,y)$ is $\mathbb{C}^2$-generic over $V$ modulo the isomorphism $\Theta=\Theta_T\times \Theta_S$.
\end{enumerate}\end{proposition}

\begin{proof}For (1), note that the sets $D_n=\{(t,s): \lh(t), \lh(s)\geq n\}$ are open dense in $\mathbb{A}_{\sa_2,(T,S)}$.\:  
 
Now we prove (2).   Let $(x,y)\in [T_G]\times[S_G]$ and consider $D\subset \mathbb{C}^2$ open dense.  By using Proposition \ref{denseproduct} and following the proof of Lemma \ref{perfect}, we have the result. \end{proof}

\begin{lemma}\label{denseproduct2}Let $n<\w$ and suppose that for every $x\in{}^\w\w$, $\mn{n}{x}$ exists and is $\w_1$-iterable.  Then, $\CC_2$ captures $\sa_2$ and $\UU_2$ over the mouse operator $\M_n^\sharp$.  \end{lemma}

\begin{proof} We will do the proof for Sacks forcing.  For Silver forcing, the procedure is completely analogous, by defining first the auxiliary forcing $\Amo{\UU}{(T,S)}:=\Amo{\UU}{T}\times\Amo{\UU}{S}$ as in \ref{sacksamo}.     Let $G$ be $\sa_2$-generic over $V$ and suppose $r\in{}^\w\w\cap {\VV[G]}$.  If $\tau$ is a $\sa_2$-name for $r$, let $(T,S)\in G$ and $\sigma\in\HC{1}$ such that $(T,S)\forc{\sa_2}{\VV}\tau=\sigma$. Let $(T',S')\leq (T,S)$.   Since $\mn{n}{T',S', \sigma}$ exists and is countable, there is $(T^*,S^*)\in V$ which is $\mathbb{A}_{\sa_2, (T',S')}$-generic over $M$.   By \ref{denseproduct}, $(T^*,S^*)\leq_{\sa_2}(T',S')$ and every pair $(p,q)\in[T^*]\times[S^*]$ is $\mathbb{C}_2$-generic over $\mn{n}{T',S', \sigma}$ modulo $\Theta$.  
So, the set of conditions \centm{$D=\{(S^*, T^*):\exists (S',T')\geq(S^*,T^*) (\text{ for all } (x,y)\in[S^*]\times[T^*]$\\\hfill$((x,y)\text{ is $\CC_2$-generic over $\mn{n}{\sigma, S',T'}$)}\}$}is dense below $(S,T)$.  Therefore, there is some $(S^*,T^*)\in D\cap G$ and $(S',T')$ witnessing it.  Thus,  the pair $(x,y)^G$ is such that $(x,y)\in[T^*]\times[S^*]$ \footnote{Assume that $G$ is $\sa_2$ generic over $V$.  Let $(x,y)^G=\cap\{[T]\times[S]: (T,S)\in G\}$.  Note that from a pair of reals added by $\sa$, we can reconstruct the generic filter associated to it as well by taking $G=\{(T,S)\in\sa_2: (x,y)\in[T]\times[S]\}$.} and there is $(T',S')\geq (T^*,S^*)$ such that\[V[G]\models (x_G, y_G)\text{ is $\mathbb{C}_2$-generic over $\mn{n}{T',S',\sigma}$ modulo $\Theta$}.\qedhere\]\end{proof}

\section{Preservation of ${M_n^\sharp}$ by forcing notions in ${\T}$}

By the results of R. David (see \cite{david}) even the simple class of $\bsi{3}$ definable c.c.c. forcings does not satisfy preservation of sharps for reals or equivalently, $\bpi{1}$-determinacy.    However, Schlicht \cite[Lemma 3.11]{phildis} has shown that every $\bsi{2}$ forcing notion which is c.c.c. in every inner model of $\ZF$,  preserves $\bpi{n+1}$-determinacy.  
Note that all forcing notions in $\T$ are $\Sigma^1_1$. 
We will show that all the forcing notions in $\T$ which are not c.c.c., but Suslin${^+}$ proper, preserve also sharps for reals.

\begin{lemma}[Lifting Lemma]\index{lifting lemma}\label{lift} Let $j:\VV\to M$ be an elementary embedding and let $\PP\in \VV$ be a forcing notion.  Suppose $G$ is $\PP$-generic over $\VV$ and $H$ is $j(\PP)$-generic over $M$.  If $j[G]\subseteq H$, then $j^*:\VV[G]\to M[H]$ defined by\centm{$j^*(\tau^G)= (\,j(\tau))^H$,}where $\tau\in\VV^\PP$, is a well-defined elementary embedding with $j^*{\upharpoonright}V=j$ and $j^*(G)=H$.  Furthermore, if $j$ is an extender ultrapower embedding, then so is $j^*$. \end{lemma}
\begin{proof}See \cite[Proposition 9.4]{cummings} for a proof in a more general setting.\end{proof}

\begin{theorem}Suppose that $\forall x\in\RR(x^\sharp$ exists$)$ and let $G$ be a $\PP$-generic filter over $V$, $\PP\in\T$. Then 
\[V[G]\models \forall x\in\RR (x^\sharp\text{\emph{\:exists}}).\]
\end{theorem}

\begin{proof}Let $y\in{}^{\w}\w$ be a real in $\VV[G]$. By Theorem \ref{keylemma} there is some $w\in{}^\w\w\cap V$ such that \[\VV[G]\models ``x_G\text{ is a }\MM\text{-generic real over }\LL[w]\footnote{Recall that the $\OR$-iterated of $w^\sharp$ by its top measure, corresponds to $\LL[w]$ (cf. Remark \ref{sharpsl}). }\text{ and $y\in L[w][x_G]$''}.\]  
  Since $w^\sharp$ exists, in $\VV$ there is a non trivial elementary embedding $j:\LL[w]\to\LL[w]$. Suppose that $\cp{j}=\kappa$.  Since $\kappa$ is greater or equal to the first $w$-indiscernible in $\LL[w]$ (cf. Lemma \ref{sheq}) we have that $\kappa>\w_1^{\LL[w]}=|\MM|^{\LL[w]}$ and, moreover $\kappa>\w_2^{\LL[w]}=|\wp(\MM)^{\LL[w]}|$.  This implies that  $j[G]=G$.  In view of Lemma \ref{lift}, the embedding  $j^*:\LL[w][x_G]\to \LL[w][x_G]$ defined by  \centm{$j^*(\sigma^{G})= (j(\sigma))^{G}$} is elementary and $\res{j^*}{L[w]}=j$, so $j^*$ is also non-trivial.   Since $y\in L[w][x_G]$ we can take $\tilde{j}:=\res{j^*}{\LL[y]}:\LL[y]\to \LL[y]$.  Therefore, $\tilde{j}$ witness that $y^\sharp$ exists in $\VV[G]$.\end{proof}

Recall the definition of $\mn{n}{x}$ in \ref{defsharps}. 
Now we will extend the results obtained in the last theorem, i.e. we will show that if $n\geq 1$, the closure of $\VV\cap\HC{1}$ under the mouse operator $x\to \mn{n}{x}$ is preserved by any forcing $\PP\in\T$.

\begin{lemma}\label{l1}Let $n\geq 1$,\: $x\in {}^\w\w$ and suppose that $\mn{n}{x}$ exists.  Then, for all $y\in\mn{n}{x}$, $\mn{n}{y}$ exists.\end{lemma}

\begin{proof}
Let $\delta_0$ be the least Woodin cardinal in $\mn{n}{x}$.  Note that $y\in\mn{n}{x}|_{\delta_0}$.  Let $M_n(x)$ be the resulting model of iterating $\mn{n}{x}$ out the universe via $U=\toe{\mn{n}{x}}$.  

\noindent We can perform inside the model $M_n(x)$ a full extender background construction over $y$ in the sense of \cite[\S 11]{fsit}, relaxing the smallness hypothesis to $n+1$-smallness.  This construction produces an eventual model $\LL[\vec{E}](y)^{M_n(x)}$ that still has $n$ Woodin cardinals. Also, by a generalization of \cite[\S 12]{fsit}, the resulting model $\LL[\vec{E}](y)^{M_n(x)}$ is $\w_1$-iterable via  iteration strategies induced by the $\w_1$-iteration strategies available for $\mn{n}{x}$.  

\noindent We have two possibilities for finding $\mn{n}{y}$ with the help of the background construction. 
\setcounter{case}{0}
\begin{case} 
If $\mn{n}{y}\li \LL[\vec{E}](y)^{M_n(x)}$, we have that $\mn{n}{y}$ is also $\w_1$-iterable and we are done.
\end{case} 
\begin{case} 
Suppose that $\LL[\vec{E}](y)^{M_n(x)}$ is $n$-small, $\kappa=\cp{U}$ and  $\tau= {\kappa^{+{\LL[\vec{E}](y)}}}^{M_n(x)}$.   Consider the structure \centm{$\mathcal{H}=\langle\LL[\vec{E}](y)^{M_n(x)}|_\tau; \in, U\cap\LL[\vec{E}](y)\rangle$.}
Observe that the potential premouse $\mathcal{H}$ satisfies the initial segment condition.  Thus, by \cite[Section 2]{nfs} we have that the $y$-premouse $\mathcal{H}$ inherits the iterability from $\mn{n}{x}$ and it is not $n$-small.  Since $\LL[\vec{E}](y)^{M_n(x)}$ is $n$-small we have that  $\mn{n}{y}=\Trcl({\hull{\mathcal{H}}{\{y\}}})$.\qedhere\end{case}\end{proof}

We rephrase an instance of \cite[Corollary 6.14]{steel} as stated in \cite[Lemma 3.5]{busche}.  This is a consequence of the Branch Uniqueness Theorem due to Steel (cf. \cite[Theorem 6.10]{steel}). 

\begin{lemma}\label{buniq} Suppose $\M$ is a tame, $k$-sound $A$-premouse which projects to $\xi=\sup(A\cap\OR)$.  Let $\itr$ be a $k$- maximal iteration tree of limit length above $\xi$ on $\M$ which is built according to the $\Q$-structure iteration strategy.  Then there is at most one cofinal, wellfounded branch $b$ through $\itr$ such that $\Q(b,\itr)=\Q(\itr)$. \end{lemma}

Now, we have the necessary elements to proof our main result.

\begin{theorem}\label{mainthm} Let $n<\w$.  Suppose that for every $x\in{}^\w\w$, $\mn{n}{x}$ exists.\footnote{The $\mn{n}{x}$ are not automatically unique by Definition \ref{defsharps}. 
However, the assumptions imply that in $V$ all $\mn{n}{x}$ are unique, and $(2)_n$ implies that in $V[G]$ all $\mn{n}{x}$ are unique, as discussed in Remark \ref{remark on uniqueness of Mn}.} 
Let $G$ be $\PP$- or $\PP^2$-generic over $\VV$ where $\PP\in\T$.  Then the following holds: 

\begin{description}

\item[$(1)_n$]  Let $x\in ({}^\w\w)^\VV$ and suppose $\VV\models \mathcal{P}=\mn{n}{x}$.  Then also $\VV[G]\models \mathcal{P}=\mn{n}{x}$.

\item[$(2)_n$] If $x\in ({}^\w\w)^{V[G]}$, then $\VV[G]\models \mn{n}{x}$ exists.

\item[$(3)_n$] Suppose that in $\VV$, $\theta$ is a sufficiently large regular cardinal and $N\prec V_\theta$ is a countable elementary substructure. Let $\bar{N}$ be the transitive collapse of $N$ with uncollapsing map $\pi:\bar{N}\to N$.  Let  $g\in\VV$ be $\PP^{\bar{N}}$-generic over $\bar{N}$.  Then, for each $x\in ({}^\w\w)^{\bar{N}[g]}$, $\mn{n}{x}^{\bar{N}[g]}$ exists and is $\w_1$-iterable in both $\bar{N}[g]$ and in $\VV$. 

\end{description}
\end{theorem}
\begin{proof} 
We will assume that $G$ is $\PP$-generic over $V$. 
The proof for $\PP^2$ is analogous, with the uses of Theorem \ref{keylemma} replaced by Lemma \ref{keylemma for product forcings}. 

The proof is an induction on $n$.  Clearly, the case $n=0$ was already proven in Section \ref{sharpsreals}.  Let $n>1$ and assume that $(1)_{n-1}$-$(3)_{n-1}$ hold.  

(1)$_n$\: :  Let $\VV\models \mathcal{P}=\mn{n}{x}$, where $x\in{}^\w\w$.   We only have to verify that $\mathcal{P}$ is $\w_1$-iterable in $\VV[G]$.  In fact, we will show that $\pc$ is $\w_1$-iterable in $\VV[G]$ via the $\Q$-structure iteration strategy $\Sigma$ of $\mn{n}{x}$ in $\VV$. 
Towards to get a contradiction, suppose that $\Sigma$ does not define an $\w_1$-iteration strategy in $\VV[G]$ for $\mathcal{P}$.  Thus, there is some iteration tree $\itr$ on $\mathcal{P}$ according to $\Sigma$ such that does not have a unique cofinal branch $b$ with $\qq{b,\itr}=\qq{\itr}$.  We can assume that there is no such cofinal branch in $\VV[G]$ by Lemma \ref{buniq}.  Then \centm{$\VV[G]\models \itr$ is a tree witnessing that $\pc$ is not $\w_1$-iterable via $\Sigma$.} Consequently, there is a $p\in G$ and a name $\dot{\itr}$ such that\centm{$p\forc{\PP}{\VV}\dot{\itr}$ is a tree witnessing that $\check{\pc}$ is not $\w_1$-iterable via $\Sigma$.} We work in the ground model $\VV$.  Let $\theta$ be large enough such that $N\prec\VV_\theta$ is a countable substructure with $\PP,\, p,\, \dot{\itr},\,\check{\pc}\in N$.  Let $\bar{N}=\Trcl(N)$ with uncollapsing map $\pi:\bar{N}\to N$.  Since $p,\,\check{\pc}\in \HHH{\w_1}$ then $p,\,\check{\pc}$ are in $\bar{N}$.  Also,  $\res{\pi}{\check{\pc}}=\id_{\check{\pc}}$.  
Now, let $\bar{\PP}$, $\bar{\itr}$ be the preimage under $\pi$ of $\PP$ and $\dot{\itr}$ respectively. 

Let $g\in\VV$ be $\bar{\PP}$-generic over $\bar{N}$ with $p\in g$.  Then \begin{equation}\label{eq20}\bar{N}[g]\models \bar{\itr}^g\text{ is a tree witnessing that $\pc$ is not $\w_1$-iterable via $\Sigma$.}\end{equation}
By $(3)_{n-1}$, we have that $\bar{N}[g]\cap\HC{1}$ is closed under the operator $x\mapsto \mn{n-1}{x}$.  Thus, if $\alpha<\lh(\bar{\itr}^g)$ we have that $\mn{n-1}{\M(\res{\bar{\itr}^g}{\alpha})}$ exists in $\bar{N}[g]$ and is $\w_1$-iterable via $\Sigma$ in $\bar{N}[g]$ and in $\VV$.   Recall that $\M(\res{\bar{\itr}^g}{\alpha})$ stands for the common part model of the tree $\bar{\itr}^g$ up to $\alpha$.

If $\Q(\res{\bar{\itr}^g}{\alpha})$ denotes the $\Q$-structure for $\res{\bar{\itr}^g}{\alpha}$ in $\bar{N}[g]$, we can compare it against $\mn{n-1}{\M(\res{\bar{\itr}^g}{\alpha})}$ by \cite[Lemma 2.36]{phildis}.  \:   Since $\Q(\res{\bar{\itr}^g}{\alpha})$ is a $\Q$-structure, ultimately we have \: $\Q(\res{\bar{\itr}^g}{\alpha})\li \mn{n-1}{\M(\res{\bar{\itr}^g}{\alpha})}$.  Thus, $\bar{\itr}$ is truly according to $\Sigma$ in $\bar{N}[g]$ and $\VV$. If $\lh(\bar{\itr}^g)$ is a  successor, we have that in $\bar{N}[g]$ the last model associated to this iteration tree on $\pc$ is well-founded, which contradicts the statement \eqref{eq20}.  Hence,  $\lh(\bar{\itr}^g)$ is a limit ordinal.\:   

Let $b=:\Sigma^\VV(\bar{\itr}^g)$ be the unique cofinal branch given by $\Sigma$ through $\bar{\itr}^g$ in $\VV$. \:   Set\begin{align*}\varphi(c)\equiv  \: \:\exists c [&c\text{ is a cofinal branch through $\bar{\itr}$ and there exists $\Q\li\M_c^{\dot{\itr}}$}\\
 &\: \: \text{such that $\Q\li\mn{n-1}{\M(\bar{\itr}^g)}$ and $\Q\models \delta(\bar{\itr}^g)$ is not Woodin]. }
 \end{align*} 
 Since $\bar{\itr}\in\HC{1}^{\bar{N}[g]}$, the formula $\varphi(\cdot)$  is $\Sigma^1_1(\vec{a}\,)$, where $\vec{a}$ is a vector formed by real codes for $\bar{\itr}^g$ and $\mn{n-1}{\M(\bar{\itr})^g}$. 
 
We have seen that $\varphi(b)$ is true in $\VV$ for $b=\Sigma^\VV(\bar{\itr}^g)$. 
We now argue that $b\in \bar{N}[g]$. 
To this end, let $g_0, g_1\in V$ be mutually generic over $\bar{N}[g]$ for a forcing that makes $\bar{\itr}^g$ countable. 
Since $\bar{N}[g][g_i]\prec_{\Sigma^1_1}\VV$ for $i\in\{0,1\}$,  we have that $\bar{N}[g]\models \varphi(b_i)$ for some $b_i\in\bar{N}[g][g_i]$.  Now, as  $b_i$ is also in $\VV$, by \ref{buniq} we have $\Sigma^\VV(\bar{\itr}^g)=b_0=b_1$. 
Now let $b=b_0=b_1$. 
Since $g_0$ and $g_1$ are mutually generic, $b\in \bar{N}[g]$. 
This contradicts our assumption in \eqref{eq20}.

(2)$_n$:\:\:  Let $x\in ({}^\w\w)^{\VV[G]}$.  By \ref{keylemma}, there is some $y\in {}^\w\w\cap V$ such that \centm{$V[G]\models x_G \text{ is $\MM$-generic over $\mn{n}{y}$ modulo $\Theta$ and $x\in\mn{n}{y}[x_G]$}$,}where $x_G$ is the $\PP$-generic real which codes $G$ and $\Theta$ is a recursive function on $y$.   Since $y\in\VV$,  by $(1)_{n}$ we conclude that $\mn{n}{y}$ exists and it is iterable in $\VV[G]$ via the $\Q$-structure iteration strategy $\Sigma$.   Therefore, to prove the existence (and thus $\w_1$-iterability) of $\mn{n}{x}$ in $\VV[G]$,  by the proof of Lemma \ref{l1} it suffices to see that $\mn{n}{y}[x_G]$ is $\omega_1$-iterable in $\VV[G]$.

Suppose that $\mn{n}{y}=\langle J_{\lambda}(y); \in, \res{\vec{E}}{\lambda}, E_\lambda\rangle$ and let $\delta_y$ be the least Woodin cardinal in $\mn{n}{y}$, $\lambda>\delta_y$ and $\cp{E_\lambda}>\delta_y$. 

Inside $\mn{n}{y}$, since $|\MM|=\w_1$, we have that all the Woodin cardinals in $\mn{n}{y}$ remain to be Woodin in the generic extension $\mn{n}{y}[x_G]$.   In fact, $\MM$ does not add any new Woodin cardinals to $\mn{n}{y}$ (cf.  \cite{hamkins}).  Now, as $|\mathbb{M}|<\cp{E}$ for every $E\in\res{\vec{E}}{\lambda}^\frown E_\lambda$,  we can lift the extender sequence $\res{\vec{E}}{\lambda}$ and the top extender $E_\lambda$ to some $\vec{E}^*$ and $E^*_{\lambda^*}$ respectively in $\mn{n}{y}[x_G]$ (cf. \cite[Proposition 9.4]{cummings}). \:  By $(1)_n$, we have that $\mn{n}{y}$ is $\w_1$-iterable in $V[G]$ via the $\Q$-structure strategy $\Sigma$.  With this in hand, it is straightforward to see that the resulting structure \centm{$\mathcal{N}=\langle J_{\lambda}(y)[x_G]; \in, \res{\vec{E}^*}{\lambda^*}, E_{\lambda^*}\rangle $} is also $\w_1$-iterable in $\VV[G]$ via the $\Q$-structure strategy $\Sigma$.   
Now, let $\mathcal{N}^*$ be the resulting model of iterating $\mathcal{N}$ out the universe (in this case, our universe is $\VV[G]$) via $U=\toe{\mathcal{N}}$.  Inside $\mathcal{N}^*$ we can perform a full extender background construction over $x$ as in \cite[\S 11]{fsit} to construct $\mn{n}{x}$.  By the proof of Lemma \ref{l1}, arguing in $\VV[G]$, $\mn{n}{x}$ inherits the Woodin cardinals and also the $\w_1$-iterability of $\mathcal{N}$.  Thus, $\mn{n}{x}$ exists in $\VV[G]$.

(3)$_n$:\:\:  We work in $\VV$.   Since $N$ is fully elementary, the results given in Theorem \ref{keylemma} holds inside the countable model $\bar{N}$.  Thus, given $x\in ({}^\w\w)^{\bar{N}[g]}$, there is some $y\in ({}^\w\w)^{\bar{N}}$ such that \centm{$\bar{N}[g]\models x_g$  is a $\MM$-generic real over $\mn{n}{y}$ and $x\in\mn{n}{y}[x_g]$.}
By assumption, $\mn{n}{y}$ exists. 
Since $j$ is elementary and fixes $\mn{n}{y}$, the latter is $\w_1$-iterable in both $\bar{N}$ and $V$. 
Working inside $\bar{N}$, by $(1)_{n}$ we have\centm{$\bar{N}[g]\models \mn{n}{y}$ is $\w_1$-iterable via the $\Q$-structure strategy $\Sigma$.}Now, by the proof of $(2)_n$ inside $\bar{N}[g]$ we have that \centm{$\bar{N}[g]\models \mn{n}{y}[x_g]$ is $\w_1$-iterable via the $\Q$-structure iteration strategy $\Sigma$.}
It follows that $\mn{n}{y}[x_g]$ is also $\w_1$-iterable via $\Q$-structures in $V$. 
We can thus construct $\mn{n}{x}$ inside $\mn{n}{y}[x_g]$ as  in the proof of $(2)_n$, and this will be $\w_1$-iterable via  $\Q$-structures in both $\bar{N}[g]$ and $V$. 
\end{proof}

According to the results of Martin, Harrington, Neeman and Woodin 
mentioned in Theorem \ref{micedet}, we obtain:

\begin{corollary}Let $n\in\w$ and assume that\, $\boldsymbol{\Pi}^1_{n+1}$\! determinacy holds in $\VV$.  Let $\PP\in\T$ and let $G$ be $\PP$-generic over $\VV$.  Then \centm{$\VV[G]\models \:\text{Every }\, \boldsymbol{\Pi}^1_{n+1}$set is determined.}
\end{corollary}

\section{Absoluteness for tree forcing  notions under determinacy assumptions} 

It is a very well-known result that $\bsi{1}$ and $\bsi{2}$-$\PP$-absoluteness hold in every  transitive model of $\ZF$ for every set forcing $\PP$.  In general, we do not have $\bsi{3}$-absoluteness between inner models of $\ZF$.  However, Martin and Solovay \cite{martinsol} gave a scenario in which we have this degree of absoluteness.

\begin{theorem}[Martin-Solovay]  Let $M$ be an inner model of $\ZFC$.  Then, $\bsi{3}$-$\PP$-absoluteness holds for every set forcing $\PP\in M$ if and only if $X^\sharp$ exists for every set $X$ in $M$.\end{theorem}

\begin{proof}The original result appears in  \cite{martinsol} making use of the so called \emph{Martin-Solovay trees}.  An alternative proof, can be found in \cite[Theorem 1]{schtalk}.   \end{proof}
The following result, derived from the original proof of Martin and Solovay, could be stated assuming just sharps for reals (cf. \cite[Theorem 2.1]{hjorth_u2}).
\begin{theorem}Let $M, N$ be inner models of $\ZFC$ closed under sharps for reals, with $M\subset N$.  If $u_2^M=u_2^N$, where $u_2$ stands for the second uniform indiscernible,  then every $\bsi{3}$-formula is absolute between $M$ and $N$. \end{theorem}

It is natural to ask whether is it possible to generalize the Martin-Solovay theorem by assuming $\bpi{n}$-determinacy for $n\geq 2$.  In this respect, Schlicht in \cite[Lemma 3.13]{phildis} has shown the next result:

\begin{theorem}\label{cccabs} Let $n<\w$ and let $\PP$ be a $\Sigma^1_2$ provably c.c.c. forcing notion.  Assume that $\mn{n}{x}$ exists and is $\omega_1$-iterable for every $x\in{}^\w\w$.  Then, $\bsi{n+3}$-$\PP$-absoluteness holds. 
\end{theorem} 

The next theorem, shows the analogous result  to Theorem \ref{cccabs} when $\PP\in \T$. 
We would like to thank the anonymous referee for the observation that for even $n$, the proof works for all Suslin${}^+$ proper forcings. 

\begin{theorem}\label{abs} Let $n\in \w$ and $\PP\in\T$.  Suppose that $\mn{n}{x}$ exists 
for all $x\in{}^\w\w$.  Then $\bsi{n+3}$-$\mathbb{P}$-absoluteness holds.
\end{theorem}

\begin{proof} By induction on $n$. Let $\varphi(x)$ be a $\Sigma^1_{n+3}$-formula with parameters in $\VV$.   We may assume that $a\in\RR$ is the only real parameter in $\varphi$.   Let $r\in{}^\w\w$.  Then, \centm{$\varphi(r)\equiv \exists y\theta(y,r,a)$,} where the formula $\theta$ is $\bpi{n+2}$. 

Suppose that $\VV\models \varphi(r)$.  Then, there is some $y_0\in V$ such that $V\models\theta(y_0, r, a)$.  By inductive hypothesis, since $\theta$ is $\bpi{n+2}$, we have that $V^\PP\models \theta(y_0,r,a)$.  Thus, $V^\PP\models \exists y\theta(y,r,a)$ i.e., $V^\PP\models\varphi(r)$. 

    Now, assume that $\VV^\PP\models \varphi(r)$ for $r\in({}^\w\w)^\VV$ and let $G$ be $\PP$-generic over $\VV$.  Then,
\centm{$\VV[G]\models \exists y\theta(y,r,a)$}Let $b\in ({}^\w\w)^{\VV[G]}$ be a witness for the sentence above and let us take a countable $\PP$-name $\tau$ for $b$ by Proposition \ref{cnames}. 
Let $p\in\PP$ be such that $p\forc{\PP}{\VV}\theta(\tau,r,a)$.  

Since $\theta$ is $\bpi{n+2}$, by Lemma \ref{fordef}, so is the formula $\theta'(p,\tau,r,a):  p\Vdash \theta(\tau,r,a)$.

Note that $\tau, r, a$ and $p$ are in $\VV$, so by assumption, $\mn{n}{\tau,r,a,p}$ exists and is $\w_1$-iterable in $\VV$.    By Lemma \ref{corr} we have two cases:
\setcounter{case}{0}
\begin{case}    If $n$ is even, we have that\centm{$\theta'(\tau, r, a,p)\iff \mn{n}{\tau,r,a,p}\models\theta'(\tau, r, a, p)$.} Since $\mn{n}{\tau,r,a,p}$ is countable,  there is some $g\in\VV$ with $p\in g$ which is $\PP^{M}$-generic over $\mn{n}{\tau,r,a,p}$, where $\PP^{M}$ is the version of $\PP$ in $\mn{n}{\tau,r,a,p}$. Hence, as $\mn{n}{\tau,r,a,p}\models p\Vdash \theta(\tau, r, a)$, we have that \centm{$\mn{n}{\tau,r,a,p}[g]\models \theta(\tau^g, r, a)$.}By the proof of Theorem \ref{mainthm}, $\mn{n}{\tau,r,a,p}[g]$ is $\w_1$-iterable.    Again, as $\theta$ is $\bpi{n+2}$ and $\mn{n}{\tau,r,a,p}[g]$ has $n$-Woodin cardinals countable in $V$, by Lemma \ref{corr} we have that\centm{$\VV\models\theta(\tau^g,r,a)$.}\end{case}

\begin{case}   Suppose that $n$ is odd.  Remember that $V[G]\models \theta(b,r,a)$, where $r,a \in V\cap{}^\w\w$ and $b\in V[G]$.  Since $V[G]$ is closed under the operator $\M_n^\sharp$ by Theorem \ref{mainthm}, we have that $\mn{n}{b,r,a}$ is $\w_1$-iterable in $V[G]$.  Let $\delta_0$ be the minimum Woodin in $\mn{n}{b,r,a}$.   Therefore, in $V[G]$ by Lemma \ref{corr}, we have \centm{$\forc{\text{Col}(\w,\, \delta_0)}{{\mn{n}{b,\,r,\,a}}}\theta(b,r,a)$.}
Since $(b,r,a)\in V[G]$ by Theorem \ref{keylemma} there is some $x\in V$ such that $b,r,a\in \mn{n}{x}[g]$ where $g$ is a $\MM$-generic real over $\mn{n}{x}$, $g\in V[G]$.  In fact, we can say that $r, a\in\mn{n}{x}$.  Thus, 
\centm{$\forc{\text{Col}(\w,\, \delta_0)}{{\mn{n}{x}[g]}} \theta(b,r,a)$.}
Since the preceding relation holds in $V[G]$, and locally inside $\mn{n}{x}[g]$, there is some $P\in g$ such that
\centm{$P\forc{\MM}{\mn{n}{x}}\:\:\forc{\text{Col}(\w,\, \delta_0)}{{\mn{n}{x}[\dot{g}]}} \theta(\tau,\check{r},\check{a})$.}
Let $h\in V$ be $\MM$-generic over $\mn{n}{x}$ with $P\in h$. Then 
\centm{$\forc{\text{Col}(\w,\, \delta_0)}{{\mn{n}{x}[h]}} \theta(\tau^{h},r,a)$.}
Since $\mn{n}{x}[h]\in V$ is $\w_1$-iterable, it follows from Lemma \ref{corr} that $V\models \theta(\tau^{h},r,a)$, i.e. $V\models \varphi(r)$.\qedhere\end{case}

 \end{proof}

This gives a partial answer to:

\begin{question}{\rm\cite[Question 7.3]{ikegami_2}} Suppose $\bdi{2}$-determinacy holds.  Then can we prove $\bsi{4}$-$\PP$-absoluteness for each strongly arboreal, proper, provably $\Delta^1_{2}$ forcing $\PP$?\end{question}

\section{Preserving orbits of thin transitive relations}
\subsection{Preserving $u_2$ by forcings in $\T$}\label{thinexample}

We will show by means of a \emph{thin} equivalence relation that forcing with any $\PP\in\T$ does not change the value of the second uniform indiscernible $u_2$.   

\begin{definition}Let $E\neq \varnothing$ be a relation defined on $\RR$.   For each $x\in\dom{E}$,\linebreak the $E$-\emph{orbit} of $x$ is the set \index{relation!eorbx@$E$-orbit}$O_E(x)=\{y\in{}^\w\w:xEy \text{  or }yEx \}$.  In particular, if $E$ is an equivalence relation, $O_E(x)$ corresponds to the $E$-equivalence class of $x$.

\noindent We say that a relation $E$ defined on the reals is \emph{thin}\index{relation!thin} if there is no perfect set $P\subset \RR$ such that for every $x,y\in P$, $x\neq y$ implies $y\notin \OO_E(x)$\footnote{This is equivalent to say that every pair of elements in $P$ are not $E$-related.}.\end{definition}

\begin{remark}If $E$ is not an equivalence relation, it might happen that $y\notin\OO_E(x)$ and still $\OO_E(x)\cap\OO_E(y)\neq\varnothing$.  Also, note that $y\in\OO_E(x)$ if and only if $x\in\OO_E(y)$.\end{remark}
\begin{definition}[Uniform indiscernibles]\label{indis}\index{uniform indiscernibles}Assume that $x^\sharp$ exists for every $x\in{}^\w\w$.\:  Let $C_x$ be the club of indiscernibles for $L[x]$ and set  $\Next{x,\delta}:=\text{min}\{\alpha\in C_x:\alpha>\delta\}$.\:   Then, for each ordinal $\gamma$, we define 
 \[u_\gamma =\begin{cases}\sup_{x\in {}^\w\w}\Next{x,0}& \text{ if $\gamma=1$},\\
 \sup_{x\in{}^\w\w}\Next{x,u_\alpha} & \text{ if $\gamma=\alpha+1$}, \\
\sup_{\alpha\in\lambda}u_\alpha& \text{ if $\gamma$ is a limit}.
 \end{cases}\] 
 
\end{definition}

If $x\in{}^\w\w$ and we assume that $x^\sharp$ exists,   all the cardinals in $V$ are indiscernibles for $L[x]$.  Therefore, if $V$ is closed under sharps for reals we have that $u_1=\w_1$ and $u_2\leq\w_2$.\:  Moreover, we have the next theorem: 
 \begin{theorem}[Kunen-Martin]\label{u2}Assume that for all $x\in {}^\w\w$, $x^\sharp$ exists.  Then, the following are all equal:
\begin{enumerate}[(i)]\item $u_2$;
\item $\sup\{(\w_1^V)^{+ L[x]}: x\in {}^\w\w\}$; 
\item $\sup \{\alpha:  \alpha\text{ is the rank of a $\boldsymbol{\Pi}^1_1$ well-founded relation}\}$;
\item ${\boldsymbol{\delta}^1_2}=\sup\{\alpha:\exists f:{}^\w\w\to\alpha \text{ such that $f$ defines a $\boldsymbol{\Delta}^1_2$ prewellordering of\,  }{}^{\w}\w\}$.\end{enumerate}
\end{theorem}

According to the characterization of $u_2$ given in item (ii) above, we define an equivalence relation closely related to it.

\begin{proposition}\label{u2e}Suppose that for every $x\in{}^\w\w$, $x^\sharp$ exists.\:  Let $E\subset\RR\times\RR$ be defined as follows:\centm{$xEy \iff ((\omega_1^V)^{+L[x]}= (\omega_1^V)^{+L[y]})$.}Then $E$ is a thin $\Delta^1_3$ equivalence relation. \end{proposition}
\begin{proof}Notice that $xEy$ iff
\centm{\small$\exists z(x,y\leq_T z \text{ and }   z^\sharp\models \kappa^{+L[x]}=\kappa^{+L[y]})\iff \forall z(x,y\leq_T z \text{ and }   z^\sharp\models \kappa^{+L[x]}=\kappa^{+L[y]})$,} where $\kappa$ is the critical point of $\toe{z^\sharp}$ and $\leq_T$ stands for Turing reducibility.  Therefore, in the presence of sharps for reals, $E$ is a $\Delta^1_3$ equivalence relation.\medskip

\noindent Now we will show that $E$ is thin.  Suppose on the contrary, that there is a perfect set $P\subset{}^{\w}{\w}$ such that $[P]^2\subset\mathbb{R}^2\smallsetminus E$. Since $E$ is $\Delta^1_3$, the formula\centm{$\forall x,y\in P (x\neq y\implies (x,y)\in\mathbb{R}^2\smallsetminus E)$}is $\Pi^1_3$ in a parameter coding $P$.
As $V$ is closed under sharps for reals, by \cite[Lemma 3.12]{phildis} we have ${\bf \Sigma}^1_3$-absoluteness for any $\Sigma^1_2$ provably c.c.c. forcing notion.   In particular, if $c$ is Cohen generic over $V$ then	\centm{$V[c]\models [P]^2\subset \mathbb{R}^2\smallsetminus E$.}
Notice that $P$ induces a $\Delta^1_3$ well-ordering of the reals by taking \[x\prec y\text{ iff }(\omega_1^V)^{+L[\varphi(x)]}<(\omega_1^V)^{+L[\varphi(y)]},\]
where $\varphi:{}^{\w}\w\to P$ is a recursive bijection with parameters in the ground model. 

Therefore, there exists $a\in{}^{\w}\w\cap V$ and a $\Delta^1_3(a)$ definition $\phi(x,y)$ such that\begin{equation}\label{eq3}V[c]\models "\{(u,v): \phi(u,v,a)\} \text{ is a well-ordering of $\mathbb{R}$"}.\end{equation}
Let $f: {}^{\w}\w\to \alpha$, $\alpha\in \textsf{OR}$ be an order-isomorphism given by (\ref{eq3}).\:  Note that $f$ is definable from the real $a\in V$.  

\noindent Thus,  $c$ is the only solution to the formula
\[\psi(x,a,\gamma): 
f(x)=\gamma\]
for some $\gamma<\alpha$, i.e. the Cohen generic real $c$ is definable with a formula using parameters from the ground model.  This is impossible, since Cohen forcing is homogeneous (see \cite[Corollary 6.63]{ralfbook}).
\end{proof}

Suppose that $\VV$ is closed under sharps for reals.  How does a forcing notion in $\T$ behave with respect to the $E$-orbits, aka $E$-equivalence clases?

\begin{theorem}\label{u2t}Let $E$ defined as in Proposition \ref{u2e} and let $\mathbb{P}\in\mathcal{T}$.\:  Suppose that $x^\sharp$ exists for all $x\in{}^\w\w$.  Then, $\PP$ does not add any new $E$-orbits.
\end{theorem}
\begin{proof} Let $G$ be $\PP$-generic over $\VV$.  It is enough to show that for every $x\in V[G]$ there exists $x'\in V$ such that $xE x'$.

\noindent Let $x\in V[G]$.  By Theorem \ref{keylemma}, there exists some $z\in {}^\w\w \cap V$ such that $x\in L[z][g]$ where $g$ is $\MM^{L[z]}$-generic over $L[z]$.\:  Since $\MM^{L[z]}$ is proper,  preserves cardinalities over $L[z]$ and thus $(\w_1^V)^{+L[x]}\leq (\w_1^V)^{+L[z,\,g]}=(\w_1^V)^{+L[z]}$.  

\noindent Suppose that $z^\sharp=\langle J_\alpha(z); \in, U\rangle$ and let $M=M_{\w_1}$ be the $\w_1$-th iterated of $z^\sharp$ by $U$.  Let $j:z^\sharp\to M$ be the induced elementary embedding.  Observe that if $\kappa:=\cp{U}=\cp{j}$ then $j(\kappa)=\w_1^V$.

\begin{figure}[h]\begin{center}
\psscalebox{1.0 1.0} 

\begin{pspicture}(0,-2.9)(6.488,2.9)
\psline[linecolor=black, linewidth=0.016](0.42799988,0.84)(0.42799988,-2.36)(0.42799988,-2.36)
\psline[linecolor=black, linewidth=0.016](4.028,2.9)(4.028,-2.36)(4.028,-2.36)
\psline[linecolor=black, linewidth=0.016](0.42799988,0.84)(0.007999878,0.84)(0.007999878,-0.5)(0.42799988,-0.5)
\psline[linecolor=black, linewidth=0.016](4.028,2.88)(3.6279998,2.88)(3.6079998,1.2576344)(4.028,1.24)(4.028,1.24)
\psline[linecolor=black, linewidth=0.016, dotsize=0.07055555cm 2.0,arrowsize=0.05291667cm 2.0,arrowlength=2.4,arrowinset=0.0]{**->}(0.3879999,-0.5)(4.068,1.26)
\rput[bl](0.107999876,-2.9){$z^\sharp[g]$}
\rput[bl](3.668,-2.86){$M[g]$}
\rput[bl](0.5879999,-0.6){$\kappa$}
\rput[bl](0.54799986,0.16){$\beta$}
\rput[bl](4.1879997,1.12){$(\omega_1^V)$}
\rput[bl](4.1679997,1.98){${(\omega_1^V)}^{+L[z]}= j(\beta)$}
\psline[linecolor=black, linewidth=0.016, linestyle=dashed, dash=0.17638889cm 0.10583334cm, dotsize=0.07055555cm 2.0,arrowsize=0.05291667cm 2.0,arrowlength=2.4,arrowinset=0.0]{**->}(0.40799987,0.42)(4.088,2.18)
\end{pspicture}

\end{center}\caption{\footnotesize{ Lifting embeddings}} \label{fig3}  
\end{figure}

\noindent We can lift $j:z^\sharp\to M$ to the extension by $\mathbb{M}^{L[z]}$ and obtain an elementary embedding $j':z^\sharp[g]\to M[g]$, see Figure \ref{fig3}.  As $x\in z^\sharp[g]$, in $z^\sharp[g]$ if  $\beta=\kappa^{+L[x]}$ we have $j(\beta)=j'(\beta)=(\w_1^V)^{+{L[x]}}$.   Thus $V[G]\models\kappa^{+L[x]}=\beta$.    Say $\gamma=j(\beta)=j'(\beta)$. 
\noindent Notice that in $\VV$, $\mathbb{M}^{L[z]}$ can be coded by a real.  Thus, the statement\begin{equation}\label{eq4}\exists x\exists g(g \text{ is $\mathbb{M}^{L[z]}$-generic over $z^\sharp$}\wedge x\in z^\sharp[g]\wedge \beta=\kappa^{+L[x]^{z^\sharp[g]}})\end{equation}
is $\Sigma^1_1$ in the parameters $z^\sharp, \beta$, $\mathbb{M}^{L[z]}$.
As (\ref{eq4}) holds in $V[G]$, by $\Sigma^1_1$-absoluteness it also holds in $V$.   If $x', g'\in V$ witness (\ref{eq4}), then $x'\in z^\sharp[g']$ and $\beta=\kappa^{+L[x']}$. 
For the lift $j''$ of $j\colon z^\sharp\to M$ to $z^\sharp[g']$, $\gamma=j(\beta)= j''(\beta)=(\w_1^V)^{+L[x']}$, i.e. $x E x'$. 
\qedhere 

\end{proof}

\begin{corollary} 
Suppose that $x^\sharp$ exists for every $x\in{}^\w\w$ and let $\mathbb{P}\in\mathcal{T}$.  Then, $\PP$ does not change the value of $u_2$, that is $u_2^V= u_2^{V^\PP}$. 
\end{corollary} 
\begin{proof}In the presence of sharps for reals, $u_2=\sup\{(\w_1^V)^{+ L[x]}: x\in {}^\w\w\}$. 
Let $G$ be $\PP$-generic over $V$ and note that $u_2=\sup\{ (\omega_1)^{+L[x]} \mid x\in ({}^\omega\omega)^V\} = \sup\{ (\omega_1^V)^{+L[x]} \mid x\in ({}^\omega\omega)^{V[G]}\} = u_2^{V[G]}$. 
\end{proof}

This gives a partial answer to: \begin{question}{\rm\cite[Question 7.4]{ikegami_2}}  Suppose every real has a sharp.  Let $\PP$ be a strongly arboreal, proper, provably ${\Delta}^1_2$-forcing.  Can we prove that every real has a sharp in $V^\PP$ and $u_2^V=u_2^{V^\PP}$?
\end{question}

\subsection{Preserving orbits of absolutely ${\bf \Delta}^1_{n+3}$ thin transitive relations}

A set is called \emph{absolutely} ${\bf \Delta}^1_n$ if it has ${\bf \Sigma}^1_n$ and ${\bf \Pi}^1_n$ definitions that remain equivalent in all generic extensions. 

\begin{lemma}\label{dt}Let  $n\in\w$ and let $E$ be a thin $\boldsymbol{\Pi}^1_{n+3}$ relation.  Suppose that $M^\sharp_n(x)$ exists and is $\w_1$-iterable for every $x\in{}^\w\w$.   Let $\mathbb{P}\in\T$ and  let $\tau_l$ be the canonical $\mathbb{P}$-name for the left generic real and  let $\tau_r$ be the canonical $\PP$-name for the generic real on the right.\:  Then, the set\centm{$D:=\{p\in\mathbb{P}:(p,p)\forc{\mathbb{P}\times\PP}{\VV}{\tau_l}\ E\tau_r\}$} is dense.\end{lemma}

\begin{proof}We follow the lines of the proof in \cite[Theorem 3.4]{fomag}.   Let $a\in{}^\w\w$ be the parameter defining the equivalence relation $E$.  Suppose that $D$ is not dense.   Then, we can pick a condition $p$ in $\mathbb{P}$ such that for every $q\leq p$, there are conditions $q_0, q_1 \leq q$ satisfying\begin{equation}\label{eq5}(q_0,q_1){\forc{\PP\times\PP}{\VV}}\neg\tau_l E\tau_r.\end{equation} Let $\theta$ be large enough such that $N\prec V_\theta$ is a countable elementary substructure with  $a,\mathbb{P},\tau_l,\tau_r,p\in N$.  Let $\bar{N}=\Trcl(N)$ with uncollapsing map $\pi:\bar{N}\to N$ and $\pi(\bar{\PP})=\PP$, $\pi(\bar{\tau}_l)=\tau_l$, $\pi(\bar{\tau}_r)=\tau_r$.  Since $p\in\HHH{\omega_1}$, $\pi(p)=p$; by elementarity between $N$ and $\VV_\theta$ we have also $E\cap N=E^N$.     

\noindent Let $\{D_n:n<\w\}$ be a enumeration of the dense open sets of \:$\bar{\mathbb{P}}\times\bar{\mathbb{P}}$ in $\bar{N}$.  Inductively,  we construct in $\bar{N}$ a tree of $\bar{\PP}$-conditions $\langle p_s: s\in{}^{<\w} 2\rangle$ satisfying 

\begin{enumerate}[(i)]  \item $p_\varnothing=p$,
\item $p_{s^\frown i}\leq p_s$ for $i\in\{0,1\}$, 
\item $(p_{s^\frown 0}, p_{s^\frown 1})\:\forc{\PP\times\PP}{\bar{N}}\neg \bar{\tau}_l\: E\:\bar{\tau}_r$,
\item $p_s$ decides the values of $\bar{\tau}_l(n)$ and $\bar{\tau}_r(n)$ for every $n<\lh(s)$, and 
\item if $s,t\in {}^i2$ and $s\neq t$, then $(p_s,p_t)\in D_0\cap\cdots \cap D_i$.
\end{enumerate}       
We begin with $s=\varnothing$.  Let \centm{$P_\varnothing=\{(t,t')\in\bar{\PP}\times\bar{\PP}: t,t'\leq p_\varnothing\text{\: and\: } (t, t')\:\forc{\bar{\PP}\times\bar{\PP}}{\bar{N}}\neg \bar{\tau}_l\: E\:\bar{\tau}_r\}$.}  By (\ref{eq5}), $P_{\varnothing}\neq\varnothing$; then, as $D_0$ is dense in $\PP\times\PP$ there is some $(p'_0,p_1')\leq (t,t')$ in $D_0$ where $(t,t')\in P_\varnothing$.   As the set of conditions \centm{$M_{0}=\{t\in \bar{\PP}: t \text{ decides the value of } \bar{\tau}(0)\}$} is open dense in $\bar{\PP}$, we can find $p_i\leq p'_i$ in $M_0$, for $i=0,1$.  Then, the condition $(p_0,p_1)\in\bar{\PP}\times\bar{\PP}$ satisfies (ii) through (v) above.

\noindent Now, assume that for every $s\in{}^{<\w}2$ of length $\leq n=\lh(s)$, we have produced $p_s$ satisfying (ii)-(v).  Then, let \centm{$P_s=\{(t,t')\in\PP\times\PP: t,t'\leq p_s\text{\: and\: } (t, t')\:\forc{\PP\times\PP}{\bar{N}}\neg \bar{\tau}_l\: E\:\bar{\tau}_r\}$.}  As $P_s\neq\varnothing$ and $D_0\cap\cdots \cap D_{n+1}$ is dense in $\PP\times\PP$, we can find $(p'_{s^\frown0},p'_{s^\frown1})\leq (t,t')$ in $D_0\cap\cdots D_{n+1}$ for some $(t,t')\in P_s$.  Now, as the set of conditions \centm{$M_n=\{ t\in\bar{\PP}:\text{ $t$ decides the values }\bar{\tau}(0),\dots,\bar{\tau}(n)\}$} is dense in $\bar{\PP}$, we can find $p_{s^\frown0}\leq p'_{s^\frown0}, p_{s^\frown1}\leq p_{s^\frown1}'$ such that $p_{s^\frown0}, p_{s^\frown1}\in M_n$.  Note that the condition $(p_{s^\frown0}, p_{s^\frown1})$ satisfies (ii)-(v), as required. 

\noindent For each $x\in{}^\w2$, $x\in\VV$, let $g_x=\{q\in\bar{\PP}: \exists n\in\w \:(p_{\res{x\,}{\,n}}\leq q)\:\}$.  Then, given two different reals $x,y\in{}^{\w}2$ in $\VV$, we get two mutually $\bar{P}$-generic reals $g_x$, $g_y$ over $\bar{N}$.  By the property (iii) of the construction, we have:\begin{equation}\label{eq6}\bar{N}[g_x,g_y]\models\neg \bar{\tau}^{g_x} E\bar{\tau}^{g_y}.\end{equation}Consider the model $\bar{N}[g_x,g_y]$.  As $g_x$ and $g_y$ are coded by the reals $x$ and $y\in\VV$, we are under the hypothesis of $(3)_n$ of Theorem \ref{mainthm}.  Therefore, $\bar{N}[g_x,g_y]$ is closed under the operator $r\mapsto \mn{n}{r}$, $r\in{}^\w\w$ and computes each $\mn{n}{r}$ correctly.  Hence, from Lemma \ref{corr} we have $\bar{N}[g_x,g_y]\prec_{\bsi{n+2}}V$.  Thus, since (\ref{eq6}) is a $\bsi{n+3}(a)$ formula, we have \centm{$\VV\models \neg\tau^{g_x}E\tau^{g_y}$} whenever $x\neq y$.\:   By the procedure generating the tree of conditions $\langle p_s:s\in{}^{<\w}2\rangle$, we have that  $x\mapsto \tau^{g_x}$ is a continuous function on ${}^\w2$.  So, we get a perfect set of pairwise pairwise not $E$-related reals.  This contradicts that $E$ is thin.\end{proof}

The next theorem improves, in some sense, the results obtained in \cite[Theorem 3.9]{phildis}. 

\begin{theorem}\label{classes}Let $E$ be a thin absolutely ${\bf \Delta}^1_{n+3}$ transitive relation\footnote{It suffices to assume that the ${\bf \Delta}^1_{n+3}$-definition remains valid in all generic extension for $\PP$ and $\PP^2$.}  and let $\PP\in\T$.  Suppose that $M^{\sharp}_n(x)$ exists for every $x\in{}^\w\w$.\:  Then $\mathbb{P}$ does not add any new $E$-orbits.	\end{theorem}	

\begin{proof}We follow the proof of \cite[Theorem 3.9]{phildis}.  Suppose that $a\in{}^\w\w$ is the defining parameter of $E$.  By Theorem \ref{abs}, $\bsi{n+3}$-$\PP$-absoluteness holds.  Thus, $E^{\VV^\PP}\cap V=E$ and we shall use $E$ to represent the set given by the same $\Sigma^1_3$ and $\Pi^1_3$-formulas defining $E$ in any $\PP$-generic extension.  Therefore, in $\VV^\PP$, $E$ is also a transitive relation. 

\noindent Let $G$ be $\PP$-generic and suppose that there is some $y\in\VV[G]\cap{}^\w\w$ such that for every $x\in{}^\w\w$, $y\notin\OO_E(x)$.  Then, if $\tau\in\HC{1}$ is a $\PP$-name for $y$, there is some $p\in\PP$ such that for every $x\in{}^\w\w$, we have \begin{equation}\label{eq7}p\forc{\PP}{\VV}\neg \check{x}E\tau.\end{equation}
By Lemma \ref{dt}, we can find a condition $q\leq p$ such that \begin{equation}\label{eq8}(q,q)\forc{\PP\times\PP}{\VV}\tau_lE\tau_r.\end{equation}Let $\theta\gg \PP$.  Take $N\prec\HHH{\theta}$ be a countable substructure such that $a,\PP, q,\tau\in N$.  As $\PP$ is proper,  there is a $(N,\PP)$-generic condition $r\leq_\PP q$. 

\noindent Let $\bar{N}=\Trcl(N)$ with uncollapsing map $\pi:\bar{N}\to N$ and let  $\pi(\bar{\PP})=\PP$.   Since $q$ and $\tau$ are in $\HC{1}$, we have that both are in $\bar{N}$ and so $\pi(q)=q$, $\pi(\tau)=\tau$.  

\noindent Pick $G_0\in \VV$ $\bar{\PP}$-generic over $\bar{N}$ with $q\in G_0$.  Now, let $G'$ be $\PP$-generic over $\VV$ with $r\in G'$ and let $G_2:=\pi^{-1}[G'\cap\PP^N]$. Note that $\pi[G_2]=G'\cap\PP^N$ is $\PP$-generic over $N$.  Also, since $r\in G'$ and $r\leq q$, we have that $q\in G'$.  So,  $q\in G'\cap \PP^N$ and therefore, $q=\pi(q)\in\pi^{-1}[G'\cap\PP^N]=G_2$.    
\begin{claim} $G_2$ is $\bar{\PP}$-generic over $\bar{N}$. \end{claim}
\begin{proof}Otherwise there is a dense set $D\in\bar{N}$, such that $G_2\cap D=\varnothing$.  Note that $\pi[D]$ is also dense in $N$; in fact, since $\PP\subset \HC{1}$, we have that every condition in $\PP^N$ is actually in $\PP^{\bar{N}}=\bar{\PP}$.   Therefore, $\pi[D]\cap G'\cap\PP^N\neq\varnothing$.  So, there is some $r'\in\pi[D]\cap(G\cap\PP^N)$.  Then, $r'\in D\cap G_2$ which contradicts our assumption.\end{proof}
\noindent Now, pick $G_1\in\VV[G']$ with $q\in G_1$ so that $G_1$ is $\bar{\PP}$-generic over $\bar{N}[G_0]$ and $\bar{N}[G_2]$.   Thus, the pairs $(G_0, G_1)$ and $(G_1,G_2)$ are $\bar{\PP}\times\bar{\PP}$-generic over $\bar{N}$.\:  Working in $\VV[G']$, set \:\: $x_i:=\tau^{G_i}$\: for each $i\in\{0,1,2\}$.
 
\noindent By \eqref{eq8}, we have also $(q,q)\forc{\bar{\PP}\times\bar{\PP}}{\bar{N}}\tau_l E\tau_r$.  Then, as $q\in G_i$, $i=0,1,2$ the following holds: 
\begin{equation}\label{eqr8}\bar{N}[G_0,G_1]\models \tau^{G_0}E\tau^{G_1}\equiv x_0Ex_1,\end{equation}
\begin{equation}\label{eq9}\bar{N}[G_1,G_2]\models \tau^{G_1}E\tau^{G_2}\equiv x_1Ex_2.\end{equation}
By Theorem \ref{mainthm}, part $(3)_n$,  we have that $\bar{N}[G_i,G_{i+1}]$, $i=0,1$, is closed under the operator $x\mapsto \mn{n}{x}$, $x\in{}^\w\w$ and computes it correctly.  Hence, by Lemma \ref{corr} we have \centm{$\bar{N}[G_0,G_1]\prec_{\bsi{n+2}} \VV[G']$\:  and\:  $\bar{N}[G_1,G_2]\prec_{\bsi{n+2}} \VV[G']$.}  Thus, it follows that\centm{$\VV[G']\models x_0Ex_1 \text{   and   }x_1Ex_2$.}
Since $E$ is transitive, we have that $\VV[G']\models x_0Ex_2$; however, notice that $x_0\in\VV$ and $x_2=\tau^{G'}\in\VV[G']$.  Since $q\in G'$, so is $p\geq q$ but then, by (\ref{eq7}) it must follow that \centm{$\VV[G']\models \neg x_0 E x_2$,} which is a contradiction.\end{proof}

\section{Open Questions}

Various arguments in this paper work for absolutely Axiom A forcings on the reals with sufficiently simple definitions. However, we do not know of any argument using only properness. Thus it is open whether the following analogues to Theorems \ref{abs} and \ref{classes} hold. 


\begin{question} 
Assume $\bpi{n+1}$-determinacy and suppose that $\PP$ is a $\bsi{2}$ proper forcing. 
\begin{enumerate}[(a)] 
\item 
Does $\bsi{n+3}$-$\PP$-absoluteness hold? 
\item 
Does $\PP$ add no new orbits to any absolutely ${\bf \Delta}^1_{n+3}$ thin transitive relation? 
\end{enumerate} 
\end{question} 

Furthermore, our proofs for thin transitive relations in Section \ref{thinexample} and related arguments in \cite[Theorem 2.3]{hjorth_thin} and \cite[Theorem 3.4]{fomag} use transitivity in an essential way. This suggests to ask whether this assumption can be removed. 

\begin{question} 
Assuming $\bpi{n+1}$-determinacy, can some $\PP\in \T$ add new connected components to some absolutely ${\bf \Delta}^1_{n+3}$ thin symmetric graph? 
\end{question} 

\section*{References}

\bibliography{ref1}

\begin{thebibliography}{10}
\expandafter\ifx\csname url\endcsname\relax
  \def\url#1{\texttt{#1}}\fi
\expandafter\ifx\csname urlprefix\endcsname\relax\def\urlprefix{URL }\fi
\expandafter\ifx\csname href\endcsname\relax
  \def\href#1#2{#2} \def\path#1{#1}\fi

\bibitem{jen9}
R.~B. Jensen, A new fine structure for higher core models, circulated
  manuscript (1997).

\bibitem{woodin}
P.~Koellner, H.~Woodin, Large cardinals from determinacy, in: M.~Foreman,
  A.~Kanamori (Eds.), Handbook of Set Theory, Vol.~3, Springer, 2010, pp.
  1951--2119.

\bibitem{gray}
C.~W. Gray, Iterated forcing from the strategic point of view, Ph.D. thesis,
  University of California, Berkeley (USA) (1970).

\bibitem{bagaria}
J.~Bagaria, R.~Bosch, Proper forcing extensions and {S}olovay models, Arch.
  Math. Logic 43~(6) (2004) 739--750.

\bibitem{fomag}
M.~Foreman, M.~Magidor, Large cardinals and definable counterexamples to the
  continuum hypothesis, Ann. Pure Appl. Logic 76 (1995) 47--97.

\bibitem{busche}
D.~Busche, R.~Schindler, The strength of choiceless patterns of singular and
  weakly compact cardinals, Ann. Pure Appl. Logic 159 (2009) 198--248.

\bibitem{cummings}
J.~Cummings, Iterated forcing and elementary embeddings, in: M.~Foreman,
  A.~Kanamori (Eds.), Handbook of Set Theory, Vol.~2, Springer, 2010, pp.
  775--883.

\bibitem{david}
R.~David, A ${\Pi}^1_2$ singleton with no sharp in a generic extension of
  $\textrm{L}$, Israel J. Math. 31 (1978) 343--352.

\bibitem{devlin}
K.~J. Devlin, Constructibility, Perspectives in Mathematical Logic,
  Springer-Verlag, Berlin, 1984.

\bibitem{fsit}
W.~J. Mitchell, J.~R. Steel, Fine structure and iteration trees, Lecture notes
  in Logic, Springer-Verlag, Berlin, New York, 1994.

\bibitem{fsk}
V.~Fischer, S.~Friedman, Y.~Khomskii, Cicho\'{n}'s diagram, regularity
  properties and $\boldsymbol{\Delta}^1_3$ sets of reals, Arch. Math. Logic 53
  (2014) 695--729.

\bibitem{gesc}
S.~Geschke, S.~Quickert, On {S}acks {F}orcing and the {S}acks {P}roperty, in:
  B.~L{\"o}we, W.~Malzkorn, T.~R\"{a}sch (Eds.), Foundations of the Formal
  Sciences II: Applications of Mathematical Logic in Philosophy and
  Linguistics, Vol.~23 of Trends Log. Stud. Log. Libr., Kluwer Acad. Publ.,
  Dordrecht, 2004, pp. 1--49.

\bibitem{gri}
S.~Grigorieff, Combinatorics on ideals and forcing, Ann. Math. Logic 3 (1971)
  363--394.

\bibitem{had}
L.~Harrington, Analytic determinacy and $0^\sharp$, J. Symb. Log. 43 (1978)
  685--693.

\bibitem{halbeisen}
L.~Halbeisen, Combinatorial set theory, Springer monographs in Mathematics,
  Springer-Verlag, London, 2012.

\bibitem{halbeisen2}
L.~Halbeisen, A playful approach to {S}ilver and {M}athias forcings, in:
  S.~Bold, B.~L{\"o}we, T.~R{\"a}sch, J.~van Benthem (Eds.), Foundations of the
  Formal Sciences V: Infinite Games, Vol.~11 of Studies in Logic, College
  Publications, London, 2007, pp. 123--142.

\bibitem{hamkins}
J.~Hamkins, H.~Woodin, Small forcing creates neither strong nor {W}oodin
  cardinals, Proc. Amer. Math. Soc. 128~(10) (2000) 3025--3029.

\bibitem{harrington2}
L.~Harrington, S.~Shelah, Counting equivalence classes for
  co-$\kappa$-{S}ouslin equivalence relations, in: J.~S. D.~van Dalen,
  D.~Lascar (Ed.), Logic Colloquium '80, North-Holland, 1982.

\bibitem{hjorth_thin}
G.~Hjorth, Thin equivalence relations and effective decompositions, J. Symb.
  Log. 58~(4) (1993) 1153--1164.

\bibitem{hjorth_u2}
G.~Hjorth, The size of the ordinal $u_2$, J. London Math. Soc. 2~(52) (1995)
  417--433.

\bibitem{ikegami_2}
D.~Ikegami, Forcing absoluteness and regularity properties, Ann. Pure Appl.
  Logic 161 (2010) 879--894.

\bibitem{jech}
T.~Jech, Set Theory, Springer Monographs in Mathematics, Springer-Verlag, 2003.

\bibitem{jensen2}
R.~B. Jensen, H.~Johnsbraten, A new construction of a ${\Delta}^1_3$
  non-constructible subset of $\omega$, Fund. Math. 81 (1974) 279--290.

\bibitem{judah}
T.~Bartoszy\'{n}ski, H.~Judah, Set theory: on the structure of the real line, A
  K Peters, Ldt., 1995.

\bibitem{kanamori}
A.~Kanamori, The higher infinite. large cardinals in set theory from their
  beginnings, Springer-Verlag, Berlin, 2003.

\bibitem{kechris1}
A.~S. Kechris, On projective ordinals, J. Symb. Log. 39~(2) (1974) 269--282.

\bibitem{laver}
R.~Laver, On the consistency of {B}orel's conjecture, Acta Math. 137 (1976)
  151--169.

\bibitem{levysol}
A.~{L}\`{e}vy, R.~M. Solovay, Measurable cardinals and the continuum
  hypothesis, Israel J. Math. 5 (1967) 234--248.

\bibitem{martinsol}
D.~A. Martin, R.~M. Solovay, A basis theorem for $\boldsymbol{\Sigma}^1_3$ sets
  of reals, Ann. Math. 89~(1) (1969) 138--159.

\bibitem{martinsteel}
D.~A. Martin, J.~Steel, Iteration trees, J. Amer. Math. Soc. 7~(1) (1994)
  1--73.

\bibitem{mathias}
A.~R.~D. Mathias, Happy families, Ann. Math. Logic 12 (1977) 59--111.

\bibitem{md}
D.~A. Martin, Measurable cardinals and analytic games, Fund. Math. 66 (1970)
  287--291.

\bibitem{miller}
A.~W. Miller, Rational perfect set forcing, Contemporary Mathematics 31 (1984)
  143--159.

\bibitem{nd}
I.~Neeman, Optimal proofs of determinacy, Bull. Symb. Logic 1~(3) (1995)
  327--339.

\bibitem{nfs}
I.~Neeman, G.~Fuchs, R.~Schindler, A criterion for coarse iterability, Arch.
  Math. Logic 49 (2010) 447--467.

\bibitem{phildis}
P.~Schlicht, Thin equivalence relations and inner models, Ann. Pure Appl. Logic
  165~(10) (2014) 1577--1625.

\bibitem{ralfbook}
R.~Schindler, Set theory. Exploring independence and truth, Springer-Verlag,
  Berlin, 2014.

\bibitem{sacksf}
G.~Sacks, Forcing with perfect closed sets, in: Axiomatic Set Theory, Vol.~13
  of Proceedings of Symposia in Pure Mathematics, Part 1, American Mathematical
  Society, Providence, 1971, pp. 331--355.

\bibitem{schtalk}
R.~Schindler, The r\^{o}le of absoluteness and correctness, slides, available
  at \url{https://ivv5hpp.uni-muenster.de/u/rds/talks-barcelona.dvi} (2003).

\bibitem{scott}
D.~Scott, Measurable cardinals and constructible sets, Bull. Pol. Ac. Math. 9
  (1961) 521--524.

\bibitem{solovay}
R.~M. Solovay, A non constructible ${\Delta}^1_3$ set of integers, Trans. Am.
  Math. Soc. 127 (1967) 50--75.

\bibitem{steel}
J.~R. Steel, An outline of inner model theory, in: M.~Foreman, A.~Kanamori
  (Eds.), Handbook of Set Theory, Vol.~3, Springer, 2010, pp. 1596--1684.

\bibitem{steelm}
J.~R. Steel, Derived models associated to mice, in: C.~Chong, Q.~Feng,
  T.~Slaman, W.~H. Woodin, Y.~Yang (Eds.), Computational {P}rospects of
  {I}nfinity, Part I: {T}utorials, Vol.~14, World Scientific Publishing
  Company, 2008, pp. 105--194.

\bibitem{stn}
J.~R. Steel, An introduction to iterated ultrapowers, in: Forcing, Iterated
  Ultrapowers and Turing Degrees, Lecture Notes Series, Institute for
  Mathematical Sciences National University of Singapore, World Scientific
  Publishing Company, Singapore, 2016, pp. 123--174.

\bibitem{ut}
S.~Uhlenbrock, Pure and hybrid mice with finitely many {W}oodin cardinals from
  levels of determinacy, Ph.D. thesis, Westfalische Wilhelms-Universit\"{a}t
  M\"{u}nster (2016).

\bibitem{woodin2}
W.~H. Woodin, On the consistency stregth of projective uniformization, in:
  J.~Stern (Ed.), Proceedings of the Herbrand Symposium: Logic colloquium 81',
  North-Holland, 1982, pp. 365--384.

\bibitem{zapletal}
I.~Neeman, J.~Zapletal, Proper forcing and $\rm{L}(\mathbb{R})$, J. Symb. Log.
  66~(2) (2001) 801--810.

\end{thebibliography}

\end{document}